\documentclass{article}

\usepackage{authblk}
\usepackage[utf8]{inputenc}
\usepackage[T1]{fontenc}
\usepackage{textcomp}
\usepackage{amsmath,amssymb,amsthm}
\usepackage{lmodern}
\usepackage[a4paper]{geometry}
\usepackage{microtype}
\usepackage{caption}
\usepackage{moreverb}
\usepackage{hyperref}
\hypersetup{
    colorlinks=true,
    linkcolor=blue,
    citecolor=magenta,
    urlcolor=blue,
    pdfborder={0 0 0}
}
\usepackage{tikz}

\usepackage[font=sf, labelfont={sf,bf}, margin=1cm]{caption}

\newtheorem*{theorem*}{Theorem}
\newtheorem{theorem}{Theorem}[section]
\newtheorem{proposition}[theorem]{Proposition}
\newtheorem{lemma}[theorem]{Lemma}

\newtheorem{notation}[theorem]{Notation}

\theoremstyle{remark}
\newtheorem{remark}[theorem]{Remark}

\newcommand{\R}{\mathbb{R}}
\newcommand{\N}{\mathbb{N}}
\newcommand{\Z}{\mathbb{Z}}

\newcommand{\C}{\mathbb{C}}

\newcommand{\Fc}{\mathcal{F}}

\newcommand{\kf}{\mathfrak{K}}
\newcommand{\af}{\mathfrak{a}}
\newcommand{\nf}{\mathfrak{n}}

\newcommand{\var}{\mathrm{Var}}

\renewcommand{\P}{\mathbb{P}}
\newcommand{\E}{\mathbb{E}}

\newcommand{\scalar}[1]{\left\langle #1 \right\rangle }
\newcommand{\floor}[1]{\left\lfloor #1 \right\rfloor}
\newcommand{\ceil}[1]{\left\lceil #1 \right\rceil}

\newcommand{\eps}{\varepsilon}
\newcommand{\good}{\mathrm{good}}
\newcommand{\bad}{\mathrm{bad}}
 
\newcommand{\eye}{\mathsf{Eye}}
\newcommand{\wick}[1]{:\!#1\!:}

\renewcommand{\d}{\mathrm{d}}

\usepackage{enumitem}

\title{Derivatives of Gaussian multiplicative chaos}
\author{Antoine Jego\thanks{CEREMADE, CNRS, Université Paris-Dauphine, PSL Research University, Place de Lattre de Tassigny, 75016 Paris, France; antoine.jego@dauphine.psl.eu}}
\date {\today}

\numberwithin{equation}{section}

\begin{document}

\maketitle

\begin{abstract}
    Consider a logarithmically-correlated Gaussian field $X$ in $d$ dimensions. For all $\gamma \in (-\sqrt{2d},\sqrt{2d})$, we show that the derivatives $\frac{\partial^k}{\partial\gamma^k} \wick{e^{\gamma X_\eps}}$ of the regularised Gaussian multiplicative chaos $\wick{e^{\gamma X_\eps}}$ converge as $\eps \to 0$. By deriving optimal bounds on their growth as $k\to\infty$, we control the power expansion of $\wick{e^{\gamma X_\eps}}$ about each $\gamma\in(-\sqrt{2d},\sqrt{2d})$. This yields an alternative approach to complex Gaussian multiplicative chaos in the whole subcritical regime, based entirely on real-valued quantities.
    
    One of our key technical contributions is to provide a truncated second moment approach to the uniform integrability of the derivatives of multiplicative chaos and its associated complex variant.
\end{abstract}

\section{Introduction}

Gaussian multiplicative chaos (GMC) is a theory making sense of
\[
\wick{e^{\gamma X}} \; = \; e^{\gamma X - \frac{\gamma^2}{2}\var(X)}
\]
where $X$ is a centred logarithmically-correlated Gaussian field and $\gamma \in \C$ is a parameter (inverse temperature).
The theory for $\gamma \in \R$ was pioneered by Kahane \cite{zbMATH03960673} and revived by \cite{zbMATH06157522, zbMATH05695672, zbMATH05944325} in relation to Liouville Quantum Gravity. It has been extensively studied since then and we refer to the book \cite{zbMATH08097122} for more background.
The understanding of the case $\gamma \notin \R$ is more recent and offers new challenges, the main difficulty coming from the fact that one encounters oscillating integrals that become tedious to control.
Extending the theory to complex values of $\gamma$ is closely related to the analyticity of $\wick{e^{\gamma X}}$ with respect to $\gamma$.

The purpose of this article is to offer a new approach to the analiticity of GMC and to the construction of its complex variant based on Hermite polynomials.
We first state our main result in Section \ref{SS:main} and then compare it with the existing literature in Section \ref{SS:context}. We will explain our strategy of proof in Section \ref{SS:strategy}.

\subsection{Main result}\label{SS:main}

The sequence of Hermite polynomials $(H_k)_{k \ge 0}$ is the unique sequence of monic polynomials orthogonal in $L^2(e^{-x^2/2} \d x)$ and such that for all $k \ge 0$, the degree of $H_k$ is $k$.
They are explicitly given by
\begin{equation}\label{E:Hermite_explicit}
    H_k(x) = e^{-\frac12 \partial_x^2} x^k = k! \sum_{m=0}^{\floor{k/2}} \frac{(-1)^m}{m!(k-2m)!} \frac{x^{k-2m}}{2^m}, \qquad k \ge 0, \quad x \in \R.
\end{equation}
Section \ref{SS:Hermite} gives more background on Hermite polynomials.

\medskip
\noindent
\textbf{Wick renormalisation.}
Let $Z$ be a centred Gaussian random variable with variance $\sigma^2$. The Hermite polynomials are used to define the Wick powers and exponentials of $Z$:
\begin{equation}\label{E:def_wick1}
    :\!Z^k\!: \; = \; \sigma^k H_k\Big(\frac{Z}{\sigma}\Big), \quad k \ge 0, \qquad :\!e^{\gamma Z}\!: \; = \; e^{\gamma Z - \frac{\gamma^2}{2}\sigma^2} = \sum_{k \ge 0} \frac{\gamma^k}{k!} :\!Z^k\!:, \quad \gamma \in \C,
\end{equation}
the last equality following from \eqref{E:Hermite_generating_function} below.
In this paper, we will be interested in normalised versions of $Z^k e^{\gamma Z}$ which can be defined in three equivalent ways:
\begin{equation}\label{E:myWick}
    :\!Z^k e^{\gamma Z}\!: \; = \; \sigma^k H_k\Big(\frac{Z-\gamma \sigma^2}{\sigma}\Big) e^{\gamma Z - \frac{\gamma^2}{2}\sigma^2} = \sum_{n \ge 0} \frac{\gamma^n}{n!} :\!Z^{n+k}\!: \; = \; \frac{\partial^k}{\partial\gamma^k} \wick{e^{\gamma Z}}, \qquad k \ge 0, \quad \gamma \in \C.
\end{equation}
The second equality follows from \eqref{E:Hermite_generating_function_generalised} below, whereas the last equality follows from \eqref{E:def_wick1}.

\medskip

Let $D\subset\R^d$ be a domain and $X$ be a centred log-correlated Gaussian field on $D$ whose covariance $C(x,y)=-\log|x-y| + g(x,y)$ is such that $g \in H^s_{\mathrm{loc}}(D\times D)$ for some $s >d$.
Consider a smooth mollifier $\varphi : \R^d \to \R$, with compact support in the unit ball $B(0,1)$ and such that $\int \varphi = 1$. For $\eps \in (0,1)$, let $\varphi_\eps = \eps^{-d} \varphi (\cdot / \eps)$ and $X_\eps = X * \varphi_\eps$.

The following result is the main result of this article.

\begin{theorem}\label{T:main}
    Let $\gamma \in (-\sqrt{2d},\sqrt{2d})$ and $f:D \to \R$ be a bounded measurable function compactly supported in $D$.
    For all $k \ge 0$, the following limit exists in $L^1$:
    \begin{equation}\label{E:T_main1}
    \scalar{\wick{X^ke^{\gamma X}},f} \quad := \quad \lim_{\eps \to 0} \int_D :\!X_\eps(x)^k e^{\gamma X_\eps(x)}\!: f(x) \d x.
    \end{equation}
    Moreover, for all $\gamma' \in \C$ with $|\gamma'-\gamma|<\sqrt{d}-|\gamma|/\sqrt{2}$, the sequence
    \begin{equation}\label{E:T_main3}
        \Big( \sum_{k \ge 0} \frac{(\gamma'-\gamma)^k}{k!}\int_D :\!X_\eps(x)^k e^{\gamma X_\eps(x)}\!: f(x) \d x \Big)_{\eps>0} = \Big( \int_D \wick{e^{\gamma' X_\eps(x)}} f(x) \d x \Big)_{\eps>0}
    \end{equation}
    converges in $L^1$ to 
    \begin{equation}\label{E:T_main2}
        \sum_{k \ge 0} \frac{(\gamma'-\gamma)^k}{k!} \scalar{\wick{X^ke^{\gamma X}},f}.
    \end{equation}
\end{theorem}

The equality \eqref{E:T_main3} is justified by \eqref{E:myWick} since $:\!X_\eps(x)^k e^{\gamma X_\eps(x)}\!:$ agrees with $\frac{\partial^k}{\partial \gamma^k}\wick{e^{\gamma X_\eps(x)}}$. Note that, for a fixed $\eps>0$, the associated power series has an infinite radius of convergence.

\medskip

\noindent\textbf{Consequences.}
Let $\eye$ be the following set depicted in Figure~\ref{fig:eye}:
\begin{equation}\label{E:eye}
    \eye = \bigcup_{\gamma \in (-\sqrt{2d},\sqrt{2d})} \Big\{ \gamma' \in \C: |\gamma'-\gamma|< \sqrt{d}-\frac{|\gamma|}{\sqrt2} \Big\}.
\end{equation}
One can verify that it agrees with the usual ``eye-shaped'' domain of complex GMC: the open convex hull of the union of the interval $(-\sqrt{2d},\sqrt{2d})$ and the disc centred at 0 and radius $\sqrt{d}$.
Putting things together,
\begin{itemize}
    \item We have defined $\scalar{\wick{e^{\gamma X}},f}$ not only for $\gamma \in (-\sqrt{2d},\sqrt{2d})$, but for all $\gamma \in \eye$;
    \item The process $\gamma \in \eye\mapsto\scalar{\wick{e^{\gamma X}},f}$ possesses an analytic modification;
    \item For all $\gamma\in(-\sqrt{2d},\sqrt{2d})$ and $k\ge0$,
    \begin{equation}
        \label{E:intro1}
    \scalar{X^k\wick{e^{\gamma X}},f} =
    \frac{\partial^k}{\partial\gamma^k}\scalar{\wick{e^{\gamma X}},f}.
    \end{equation}
\end{itemize}
Finally, since we know that $\scalar{\wick{e^{\gamma X}},f}$ is universal in the sense that it does not depend on the specific choice of mollifier $\varphi$ (see e.g. \cite{zbMATH06725033}),  \eqref{E:intro1} implies that $\scalar{\wick{X^ke^{\gamma X}},f}$ is also universal.

\begin{remark}
    It should not be difficult to extend Theorem \ref{T:main} to show the convergence of $\wick{X_\eps^k e^{\gamma X_\eps}}$ in appropriate Sobolev spaces with negative indices.
\end{remark}

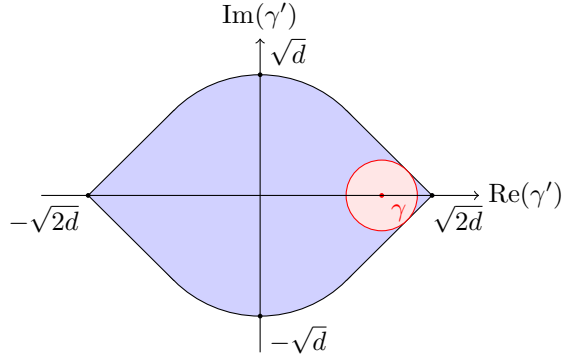
\begin{figure}
    \centering
\begin{tikzpicture}[scale=1.6]
  \pgfmathsetmacro{\d}{1}            
  \pgfmathsetmacro{\R}{sqrt(\d)}     
  \pgfmathsetmacro{\S}{sqrt(2*\d)}   
  \pgfmathsetmacro{\a}{sqrt(\d/2)}   
  \pgfmathsetmacro{\r}{\R*(1 - 1/sqrt(2))}  
  \pgfmathsetmacro{\cx}{\R}          
  \pgfmathsetmacro{\cy}{0}           

  \filldraw[fill=blue!18, draw=black] (0,0) circle (\R);
  \fill[blue!18] (\S,0) -- (\a,{\S-\a}) -- (\a,{-\S+\a}) -- cycle;
  \fill[blue!18] (-\S,0) -- (-\a,{-\S+\a}) -- (-\a,{\S-\a}) -- cycle;

  \draw[black] (\S,0) -- (\a,{\S-\a});
  \draw[black] (\S,0) -- (\a,{-\S+\a});
  \draw[black] (-\S,0) -- (-\a,{\S-\a});
  \draw[black] (-\S,0) -- (-\a,{-\S+\a});

  \filldraw[fill=red!10, draw=red] (\cx,\cy) circle (\r);
  \filldraw[red] (\cx,\cy) circle (0.015);
  \node[red, below right] at (\cx,\cy) {$\gamma$};

  \draw[->] (-1.8,0) -- (1.8,0) node[right] {$\mathrm{Re}(\gamma')$};
  \draw[->] (0,-1.3) -- (0,1.3) node[above] {$\mathrm{Im}(\gamma')$};

  \filldraw[black] (\S,0) circle (0.015);
  \node[black, below right] at ([xshift=-2pt, yshift=0pt]\S,0) {$\sqrt{2d}$};
  \filldraw[black] (-\S,0) circle (0.015);
  \node[black, below left] at ([xshift=0pt, yshift=0pt]-\S,0) {$-\sqrt{2d}$};
  \filldraw[black] (0,\R) circle (0.015);
  \node[black, above right] at ([xshift=0pt, yshift=0pt]0,\R) {$\sqrt{d}$};
  \filldraw[black] (0,-\R) circle (0.015);
  \node[black, below right] at ([xshift=0pt, yshift=0pt]0,-\R) {$-\sqrt{d}$};
\end{tikzpicture}

\caption{Illustration of the eye-shaped domain $\eye$ \eqref{E:eye} and a disc centred at some point $\gamma \in (-\sqrt{2d},\sqrt{2d})$ with radius $\sqrt{d}-|\gamma|/\sqrt2$. As $\gamma$ varies in $(-\sqrt{2d},\sqrt{2d})$, these discs cover all the set $\eye$. In this paper, we define $\wick{e^{\gamma'X}}$ in the whole set $\eye$ by defining it in each such discs.}
    \label{fig:eye}
\end{figure}

\subsection{Context}\label{SS:context}
The convergence of $(\int_D :\!X_\eps(x)^k e^{\gamma X_\eps(x)}\!: f(x) \d x)_{\eps>0}$ as stated in \eqref{E:T_main1} is well known when $k=0$: it simply corresponds to the construction of GMC with a real coupling constant $\gamma$. See \cite[Chapter 3]{zbMATH08097122} and \cite{zbMATH06370363} and the references therein.
With motivations coming from quantum field theory, the article \cite{zbMATH07500557} treated the case $k=1$, i.e. the first derivative of GMC, but was restricted to the $L^2$-phase $\gamma \in (-\sqrt{d},\sqrt{d})$.
The case $k=1$ and $\gamma=\pm\sqrt{2d}$ has also been treated in the literature, but has a different flavour. For such values of $\gamma$,
$(\int_D \wick{e^{\gamma X_\eps(x)}} f(x) \d x)_{\eps>0}$ converges to 0 in probability, while the first derivative $(\int_D \wick{X_\eps(x)e^{\gamma X_\eps(x)}} f(x) \d x)_{\eps>0}$ converges in probability to a nondegenerate measure: the critical GMC. See the survey article \cite{zbMATH07392953} for more on this topic.
Let us also mention that the case $\gamma=0$ and any value of $k \ge 0$, corresponding to Wick powers $\wick{X^k}$ of $X$, is well known. The power expansion of $\wick{e^{\gamma X}}$ about $\gamma =0$, i.e. that for all test function $f$ and $|\gamma|<\sqrt{d}$,
\[
\scalar{\wick{e^{\gamma X}},f} = \lim_{K \to \infty} \sum_{k=0}^K \frac{\gamma^k}{k!} \scalar{\wick{X^k},f} \quad \text{in }L^2,
\]
is also well known; see \cite[Section 2.3]{zbMATH06732240}.

The article \cite{zbMATH07493822} recently showed that, for all complex values of $\gamma \in \eye$ \eqref{E:eye}, the following limit exists in $L^1$:
\[
\scalar{\wick{e^{\gamma X}},f} := \lim_{\eps \to 0} \int_D :\!e^{\gamma X_\eps(x)}\!: f(x) \d x.
\]
Before \cite{zbMATH07493822}, the above limiting object had already been defined in \cite{zbMATH07172346} using a different renormalisation procedure (see also \cite{zbMATH06032744} considering a more restrictive class of log-correlated Gaussian fields and \cite{zbMATH06442803} treating a variant of the model where the real and imaginary parts of the field are independent).
This line of research was inspired by the article \cite{zbMATH02154235} which studies multiplicative cascades.
As already alluded to, as soon as one is able to define $\scalar{\wick{e^{\gamma X}},f}$ for complex values of $\gamma$, it is easy to check that $\gamma \in \eye \mapsto \scalar{\wick{e^{\gamma X}},f}$ is analytic.
Here is a quick argument which the author learnt in \cite{zbMATH07249090}.
For fixed $\eps >0$, $\gamma \in \eye \mapsto \scalar{\wick{e^{\gamma X_\eps}},f}$ is clearly holomorphic and thus, by the residue theorem, for all $\gamma \in \eye$,
\[ 
\scalar{\wick{e^{\gamma X_\eps}},f} = \frac1{2i\pi} \oint \frac1{z-\gamma} \scalar{\wick{e^{z X_\eps}},f} \d z,
\]
where the contour is any loop surrounding $\gamma$. If this loop is compactly included in $\eye$, and with some additional integrability estimate, one can apply dominated convergence theorem and deduce that
\[ 
\scalar{\wick{e^{\gamma X}},f} = \frac1{2i\pi} \oint \frac1{z-\gamma} \scalar{\wick{e^{z X}},f} \d z.
\]
The analyticity of $\gamma \in \eye \mapsto \scalar{\wick{e^{\gamma X}},f}$ then follows.

Our article offers an alternative approach to the analyticity of Gaussian multiplicative chaos and the definition of its complex variant in the whole subcritical domain $\eye$ \eqref{E:eye}.
Compared to the aforementioned literature, we take a reverse route:
we define directly the derivatives $\frac{\partial^k}{\partial\gamma^k}\scalar{\wick{e^{\gamma X}},f}$ as the limit of the regularised derivatives. By controlling the growth of these derivatives as $k\to \infty$, we then obtain optimal bounds on the radius of convergence of the power expansion of $\scalar{\wick{e^{\gamma X}},f}$ about each $\gamma \in (-\sqrt{2d},\sqrt{2d})$ allowing us to analytically extend $\scalar{\wick{e^{\gamma X}},f}$ to the entire set $\eye$.

\begin{remark}[Brownian multiplicative chaos]
In the articles \cite{zbMATH00681375, zbMATH07224960, zbMATH07224956, zbMATH07740451}, a multiplicative chaos associated to the square root of the local time of 2D Brownian motion/Brownian loop soup has been constructed for all $\gamma\in(0,2)$ (and in fact also for $\gamma=2$ in \cite{zbMATH07357903}).
The renormalised powers of the local time have also been constructed based on the generalised Laguerre polynomials; see \cite{zbMATH00219137, zbMATH05911373} and the references therein. In the context of the Brownian loop soup, \cite[Theorem~9.2]{JLQ} shows that these renormalised powers coincide with the coefficients of the power expansion of the multiplicative chaos about $\gamma=0$. 
It would be interesting to see whether this can be generalised to expansions about any value of $\gamma\in(0,2)$, as in the present paper dealing with the Gaussian case. Actually, one of the main motivations of this work was to develop an approach robust enough to be able to treat such cases; see Section~\ref{SS:strategy} below.

We mention that the article \cite{zbMATH07471411} considered some non-Gaussian log-correlated fields, constructed an associated multiplicative chaos for $\gamma$ in a complex neighbourhood of $(-\sqrt{2d},\sqrt{2d})$ and proved its analyticity there. By definition, the field possesses a nice martingale structure which enables the author to use $L^p$-estimates.
\end{remark}

\subsection{Strategy}\label{SS:strategy}

We now explain the strategy of the proof of Theorem \ref{T:main}.
This follows the same overall strategy as in the construction of real GMC (see in particular \cite{zbMATH06725033} for a streamlined presentation), but with important differences especially concerning the way we deal with the contribution of the ``bad event''.

As usual, proving the convergence of $\wick{X_\eps^k e^{\gamma X_\eps}}$ as stated in \eqref{E:T_main1} is relatively easy when $\gamma \in (-\sqrt{d},\sqrt{d})$ and is based on a second moment method. However, even for such values of $\gamma$ (except in the special case $\gamma=0$), a naive second moment method does not give the optimal radius of convergence in the series \eqref{E:T_main2}. To obtain an optimal bound, and to be able to cover the whole regime $\gamma \in (-\sqrt{2d},\sqrt{2d})$, we need to introduce truncations. For the rest of this discussion, we focus on the case $\gamma \in [0,\sqrt{2d})$. The idea of this truncation is based on the observation that the GMC measure $\wick{e^{\gamma X}}$ is almost surely supported on $\gamma$-thick points:
\[
\{ x \in D: \lim_{\eps \to 0} X_\eps(x)/|\log \eps| = \gamma \}.
\]
The value $\gamma = \sqrt{2d}$ is also natural in this context since it corresponds to the maximal possible thickness (see \cite{zbMATH05695679} in the case of the Gaussian free field):
\begin{equation}\label{E:max_thick}
    \sup_{x \in D} X_\eps(x)/|\log \eps| \xrightarrow[\eps \to 0]{\P} \sqrt{2d}.
\end{equation}
The ``good event'' at $x$ is then defined as the event that the field at $x$ is not more than $\hat\gamma$-thick, where $\hat\gamma>\gamma$ is a parameter. The rest of the proof is divided into two main steps.
\begin{itemize}
    \item \textbf{Step 1} (second moment, good event): we show that for all $k \ge 0$, $(\wick{X_\eps^k e^{\gamma X_\eps}})_{\eps >0}$ integrated against a test function and restricted to the good event is Cauchy in $L^2$.
    \item \textbf{Step 2} (first moment, bad event): we show that the restriction to the bad event is small in~$L^1$.
\end{itemize}
Step~1 requires some technical novelties compared with the standard case $k=0$, especially in the way we handle the Hermite polynomials. Step~2 differs drastically from that case: one needs to bound an expectation of the form
\[
\E\Big[ \Big| \int \wick{X_\eps(x)^k e^{\gamma X_\eps(x)}} \mathbf{1}_{\text{bad event}(x)} \d x \Big| \Big].
\]
When $k=0$, the integrand is nonnegative so the absolute values can be removed and the integral and the expectation can be exchanged.
It then boils down to bounding $\E[\wick{e^{\gamma X_\eps(x)}} \mathbf{1}_{\text{bad event}(x)}]$ which is a simple computation.
When $k \ge 1$, the integrand does not have a constant sign and the absolute values cannot be removed. The same problem appears when one considers $\wick{e^{\gamma X}}$ for $\gamma\notin\R$, so we explain in Remark~\ref{R:complex} below how this has been circumvented in this context and why those strategies cannot be used here.

We resolve this problem by deriving an inductive inequality on this first moment: by looking at the first scale where the good event fails to hold, we bound this first moment in terms of similar terms but for larger regularisation scales.
It is easier to write this argument for almost $*$-scale invariant fields (see Section \ref{S:setup} below for details about these fields) which are covariant under rescalings. We can then treat general log-correlated fields using the decomposition result of \cite{zbMATH07172346}. In its simplest form, this inductive inequality reads as
\begin{equation}
    \label{E:ineq_aa}
M_N \le \max(M_{N-1}, C_1 + C_2 M_{N-1}), \quad N\ge1,
\end{equation}
where $C_1,C_2>0$ are constants and, for $\eta >0$ small,
\begin{align*}
    M_N = \max_{\substack{\eps = e^{-n}\\0 \le n \le N}} \max_{k\ge 0} \frac{1}{k!} \Big(  \frac{(1+\eta)\sqrt{2}}{\sqrt{2d}-\gamma} \Big)^{-k} \hspace{-5pt} \sup_{\af \in [0,1]} \sup_{0<\|f\|_\infty < \infty} \frac1{\|f\|_\infty} \E \Big[ \Big| \int_{[0,1]^d} \hspace{-4.6pt} f(x) :X_\eps^\af(x)^k e^{\gamma X_\eps^\af(x)}: \d x \Big| \Big].
\end{align*}
The inequality \eqref{E:ineq_aa} then guarantees that $(M_N)_{N \ge 0}$ is uniformly bounded as soon as $C_2 <1$ and $M_0 <\infty$; see Lemma \ref{L:very_elementary}.
We believe that this argument is quite general and could be useful in other contexts. The field only needs to satisfy a weak form of scale invariance. For instance, the almost $*$-scale invariant fields are only scaling covariant; see Lemma \ref{L:scaling_covariance} below. In the above definition of $M_N$, the supremum over the scaling parameter $\af \in [0,1]$ is crucial to ``close the inequality''.

\begin{remark}[Construction of complex GMC]\label{R:complex}
We elaborate on the construction of $\wick{e^{\gamma X}}$ for complex values of $\gamma$ in the set $\eye$.
The strikingly simple observation made by \cite{zbMATH07493822} is that, in this context of complex GMC, it is actually enough to consider $\hat\gamma > \sqrt{2d}$ (close to $\sqrt{2d}$) to make Step 1 work.
Because $\sqrt{2d}$ corresponds to the maximal thickness (see \eqref{E:max_thick}), the restriction to the bad event vanishes with high probability when $\hat\gamma>\sqrt{2d}$. Hence, Step 2 is not even necessary to obtain convergence in probability. To obtain convergence in $L^1$, one needs additional estimates. Instead of considering a truncated second moment method as above, this is done using $L^p$ estimates where $p>1$ is close to 1 (depending on $\gamma$).
These estimates are rather tricky and are first obtained for $*$-scale invariant fields and their martingale approximations (see for instance \cite[Section 6]{zbMATH07172346}, or \cite{zbMATH07870787} which uses the Burkholder-Davis-Gundy inequality).
Let us also mention the work \cite{zbMATH07453036} which bounds exponential moments of Wick powers of the Gaussian free field using the Burkholder-Davis-Gundy inequality

Going back to our proof of Theorem \ref{T:main}, in principle we should be able to use the same simple trick as \cite{zbMATH07493822} and take $\hat\gamma>\sqrt{2d}$. This would allow us to only consider Step 1 above and prove that the convergence \eqref{E:T_main1} holds in probability. However, to obtain the convergence of \eqref{E:T_main3}, we need to control the truncated second moment uniformly in the power $k$ (see Proposition~\ref{P:bounded_L2} below). As we will see, these estimates will be sharp only when $\hat\gamma \in (\gamma,(2\gamma)\wedge\sqrt{2d})$ and so Step~2 cannot be skipped.
Alternatively, we could try to obtain $L^p$ estimates for $p>1$. Obtaining optimal bounds which are uniform in $k$ seems however difficult.
\end{remark}

As alluded above, our estimate on the truncated second moment, stated in Proposition \ref{P:bounded_L2}, is sharp only when $\hat\gamma\in(\gamma, (2\gamma)\wedge \sqrt{2d})$. The restriction $\hat{\gamma}< 2\gamma$, which is discussed in Remark~\ref{R:restriction}, is problematic since we actually want to take $\hat\gamma$ arbitrarily close to $\sqrt{2d}$. When $\gamma < \sqrt{2d}/2$, our second moment estimate is thus not enough to obtain the optimal radius of convergence in \eqref{E:T_main3}. Instead, we will consider moments of order $2p$ which effectively replaces the restriction $\hat\gamma<2\gamma$ by $\hat\gamma<2p\gamma$; see Proposition \ref{P:bound_pmoment}. In this way, we will be able to obtain the optimal radius of convergence for any $\gamma>0$. The case $\gamma=0$ is treated separately in Section \ref{SS:gamma=0}. As already mentioned, the case $\gamma=0$ is well known (see e.g. \cite[Section 2.3]{zbMATH06732240}) and is easier to deal with since it is the only case where the disc of convergence $\{\gamma'\in\C:|\gamma'-\gamma|<\sqrt{d}-|\gamma|/\sqrt2\}$ entirely lies in the $L^2$-phase $\{\gamma'\in\C:|\gamma'|<\sqrt{d}\}$.

\subsection{Organisation}
The rest of the article is organised as follows.
Section \ref{S:setup} introduces the setup concerning almost $*$-scale invariant fields. Section \ref{S:preliminaries} is a preliminary section gathering properties on Hermite polynomials and almost $*$-scale invariant fields.
Section \ref{S:convergence} is the core of the paper where we prove all the needed estimates about GMC and its derivatives, in the context of almost $*$-scale invariant fields. Finally, in Section \ref{S:reduction} we prove our main result Theorem \ref{T:main} by coupling the field $X$ to an almost $*$-scale invariant field.

\section{Setup}\label{S:setup}

In most of this paper, we will consider an almost $*$-scale invariant field. This sections introduces the associated setup. We refer to \cite{zbMATH07172346} for more background.
Let $\kf : \R^d \to \R$ be a seed covariance satisfying:
\begin{itemize}
\item $\kf \ge 0$ and $\kf(0)=1$.
\item $\kf$ is rotationally symmetric, i.e. $\kf(x) = \kf(y)$ for all $x,y \in \R^d$ with $|x|=|y|$.
\item $\kf$ is compactly supported in the unit ball.
\item There exists $s > (d+1)/2$, such that $0 \le \hat\kf(\xi) \lesssim (1+|\xi|^2)^{-s}$ for all $\xi \in \R^d$.
\end{itemize}
An example of such a seed covariance can be constructed as follows. Let $\phi$ be smooth, nonnegative, rotationally symmetric, supported in $B(0,1/2)$ and such that $\|\phi\|_{L^2} = 1$. Then $\kf = \phi*\phi$ satisfies the above properties.

The almost $*$-scale covariance kernel $C^\af$ associated to the seed $\kf$ and parameters $\alpha\in (0,+\infty]$ and $\af \in [0,1]$ is given by
\begin{equation}\label{E:covariance_star_almost}
C^\af(x,y) := \int_0^\infty \kf(e^u (x-y))(1- \af e^{-\alpha u}) \d u, \qquad x,y \in \R^d.
\end{equation}
The parameter $\af$ is often set to 1 in the literature. This parameter will be useful for us since it guarantees a scaling covariance property; see Lemma \ref{L:scaling_covariance} below.
When $\alpha=+\infty$ (or $\af=0$), this is referred to as $*$-scale invariant covariance kernel in the literature.
A field associated to this kernel can be constructed as follows. Let $W$ be a space-time white noise and $\tilde{\kf} = \Fc^{-1} \sqrt{\Fc \kf}$ where $\Fc$ and $\Fc^{-1}$ denote the Fourier transform and its inverse. Let then
\begin{equation}\label{E:X_WN}
X^\af(x) = \int_{\R^d \times (0,\infty)} e^{du/2} \tilde\kf(e^u(z-x)) \sqrt{1-\af e^{-\alpha u}} \d W(z,u), \qquad x \in \R^d.
\end{equation}
A natural approximation of the field $X$ is given by
\begin{equation}\label{E:X_approx_star_scale}
X_t^\af(x) = \int_{\R^d \times (0,t)} e^{du/2} \tilde\kf(e^u(z-x)) \sqrt{1-\af e^{-\alpha u}} \d W(z,u), \qquad x \in \R^d, \quad t \ge 0.
\end{equation}
Consider also a smooth mollifier $\varphi : \R^d \to \R$, with compact support in $B(0,1)$ and such that $\int \varphi = 1$. For $\eps \in (0,1)$ and $t \in (0,\infty)$, let $\varphi_\eps = \eps^{-d} \varphi (\cdot / \eps)$ and
\begin{equation}\label{E:X_approx_mollifier}
    X_\eps^\af = X^\af * \varphi_\eps
    \quad \text{and} \quad
    X_{t,\eps}^\af = X^\af_t * \varphi_\eps.
\end{equation}

For $s,t \ge 0$, $\eps,\delta \in (0,1)$ and $x,y \in \R^d$, denote by 
\begin{gather}
    C_{s,t}^\af(x,y) = \E[X^\af_s(x) X^\af_t(y)], \quad C_t^\af(x,y) = C_{t,t}^\af(x,y), \quad \sigma_t^\af = \sqrt{C_t^\af(x,x)},\\
    C_{\eps,\delta}^\af(x,y) = \E[X^\af_\eps(x) X^\af_\delta(y)], \quad C_\eps^\af(x,y) = C_{\eps,\eps}^\af(x,y), \quad \sigma_\eps^\af = \sqrt{C_\eps^\af(x,x)},\\
    C^\af_{t,\eps}(x,y) = \E[X^\af_t(x) X^\af_\eps(y)],
    \quad \sigma_{t,\eps}^\af = \sqrt{\E[X_{t,\eps}^\af(x)^2]}.
\end{gather}
When there is no ambiguity with the value of $\af$, we will drop it from our notation and simply write $X(x)$, $X_t(x)$, $X_\eps(x)$, $C(x,y)$, etc.

We now state the main result in this context.

\begin{theorem}\label{T:convergence_k}
Let $\gamma\in(-\sqrt{2d},\sqrt{2d})$.
For any bounded measurable function $f:\R^d \to \R$ with compact support and $k \ge0$,
\begin{align}\label{E:defI}
I_\eps^\af(f,k) := \int_{\R^d} f(x) :X_\eps(x)^k e^{\gamma X_\eps(x)}: \d x
\end{align}
converges in $L^1$ as $\eps \to 0$ to some random variable $I^\af(f,k)$.
Moreover, for all $\eta >0$,
\begin{align}\label{E:T_convergence_k}
    \sup_{\eps >0} \sup_{k \ge 0} \frac{1}{k!} \Big((1+\eta) \frac{\sqrt2}{\sqrt{2d}-|\gamma|} \Big)^{-k} \sup_{0<\|f\|_\infty<\infty} \frac1{\|f\|_\infty} \E[|I_\eps^\af(f,k)|] < \infty,
\end{align}
where the last supremum ranges over bounded measurable functions $f:\R^d \to \R$ with compact support included in $[0,1]^d$.
\end{theorem}

\begin{remark}
    Using the tools developed in this paper, it should not be complicated to show that for all bounded measurable function $f : \R^d \to \R$ with compact support and for all $k \ge 0$,
    \[
    \int_{\R^d} f(x) :X_t(x)^k e^{\gamma X_t(x)}: \d x
    \]
    converges in $L^1$ to $I^\af(f,k)$ as $t \to \infty$. This provides a martingale approximation to $I^\af(f,k)$. We do not write a proof of this result since it is specific to $*$-scale invariant fields.
\end{remark}

\section{Preliminaries}\label{S:preliminaries}

\subsection{Hermite polynomials}\label{SS:Hermite}

This section discusses elementary properties of the Hermite polynomials $(H_k)_{k\ge0}$ \eqref{E:Hermite_explicit}.
Some of the results of this section can be found in textbooks, but we prefer to give short derivations of the results we need. 
One of the most important properties concerns the generating function of $(H_k)_{k\ge0}$:
\begin{equation}
    \label{E:Hermite_generating_function}
    \sum_{k \ge 0} \frac{t^k}{k!} H_k (x)  = e^{tx-t^2/2}, \qquad t \in \R, \quad x \in \R.
\end{equation}
This identity can be proven directly from the explicit expression \eqref{E:Hermite_explicit} of $H_k, k \ge 0$. We will also need a slight generalisation:
    \begin{equation}\label{E:Hermite_generating_function_generalised}
        \sum_{n \ge 0} \frac{t^n}{n!} H_{k+n} (x)  = e^{tx-t^2/2} H_k(x-t), \qquad t \in \R, \quad x \in \R, \quad k \ge 0.
    \end{equation}
    To prove \eqref{E:Hermite_generating_function_generalised}, take $s \in \R$. On the one hand,
    \begin{align*}
        \sum_{k \ge 0} \frac{s^k}{k!} \Big( \sum_{n \ge 0} \frac{t^n}{n!} H_{k+n} (x) \Big)
        & = \sum_{m\ge 0} H_m(x) \sum_{n=0}^m \frac{t^n s^{m-n}}{n!(m-n)!} \\
        & = \sum_{m\ge 0} \frac{(s+t)^m}{m!} H_m(x) = e^{(s+t)x-(s+t)^2/2},
    \end{align*}
    where the last equality follows from \eqref{E:Hermite_generating_function}. On the other hand, using once more \eqref{E:Hermite_generating_function},
    \begin{align*}
        \sum_{k \ge 0} \frac{s^k}{k!} \big( e^{tx-t^2/2} H_k(x-t) \big)
        = e^{tx-t^2/2} e^{s(x-t)-s^2/2} = e^{(s+t)x-(s+t)^2/2}.
    \end{align*}
    Hence, both sides of \eqref{E:Hermite_generating_function_generalised} have the same generating function and must be equal.

\begin{lemma}\label{L:Hermite_umbral}
    For all $k \ge 0$, $u,v \in \R$ and $\rho \in [-1,1]$,
    \begin{equation}\label{E:L_Hermite_umbral1}
        H_k(\rho u + \sqrt{1-\rho^2} v) = \sum_{j=0}^k \binom{k}{j} \rho^j \sqrt{1-\rho^2}^{k-j} H_j(u) H_{k-j}(v),
    \end{equation}
    \begin{equation}\label{E:L_Hermite_umbral2}
        H_k(u+v) = \sum_{j=0}^k \binom{k}{j} u^j H_{k-j}(v),
    \end{equation}
    \begin{equation}\label{E:L_Hermite_umbral3}
        H_k(\rho u) = \sum_{i=0}^{\floor{k/2}} (-1)^i \rho^{k-2i} (1-\rho^2)^i \binom{k}{2i} \frac{(2i)!}{i!} 2^{-i} H_{k-2i}(u).
    \end{equation}
\end{lemma}

\begin{proof}
    We give a proof for completeness.
    Let $t >0$. By \eqref{E:Hermite_generating_function}, we have
    \begin{align*}
        & \sum_{k \ge 0} \frac{t^k}{k!} H_k(\rho u + \sqrt{1-\rho^2} v) = e^{(\rho u + \sqrt{1-\rho^2} v)t-\frac12t^2}
        = e^{-u(\rho t) - \frac12(\rho t)^2} e^{-v(\sqrt{1-\rho^2}t)-\frac12(\sqrt{1-\rho^2}t)^2} \\
        & = \Big( \sum_{k \ge 0} \frac{(\rho t)^k}{k!} H_k(u) \Big) \Big( \sum_{k \ge 0} \frac{(\sqrt{1-\rho^2}t)^k}{k!} H_k(v) \Big)
        = \sum_{k \ge 0} t^k \sum_{j=0}^k \frac{\rho^j \sqrt{1-\rho^2}^{k-j}}{j!(k-j)!} H_j(u) H_{k-j}(v).
    \end{align*}
    Identifying the coefficients of the expansion proves \eqref{E:L_Hermite_umbral1}. For \eqref{E:L_Hermite_umbral2}, we apply \eqref{E:L_Hermite_umbral1} to $(\rho,u,v) \leftarrow (\rho,u/\rho,v)$ and we let $\rho \to 0$. The result follows since $\rho^j H_j(u/\rho) \to u^j$ as $\rho \to 0$.
    Finally, \eqref{E:L_Hermite_umbral3} follows from \eqref{E:L_Hermite_umbral1} by considering $v=0$ and from the fact that for all $n\ge 0$, $H_{2n+1}(0)=0$ and $H_{2n}(0)=(-2)^{-n} (2n)!/n!$ (which follows from the explicit formula \eqref{E:Hermite_explicit}).
\end{proof}

\begin{lemma}\label{L:Hermite_bound}
    For all $\rho \in (0,1)$, $x \in \R$ and $k \ge 0$,
    \begin{equation}\label{E:L_Hermite_upper_bound}
        |H_k(x)| \le (1-\rho^2)^{-1/4} \rho^{-k/2} \sqrt{k!} e^{\frac{\rho}{2(1+\rho)} x^2}.
    \end{equation}
\end{lemma}

\begin{proof}
    Let $\rho \in (0,1)$.
    The Mehler's formula reads as
    \begin{align*}
        \sum_{k\ge 0} \frac{\rho^k}{k!} H_k(x) H_k(y) = (1-\rho^2)^{-1/2} \exp \Big( \frac{xy \rho - (x^2+y^2)\rho^2/2}{1-\rho^2} \Big).
    \end{align*}
    See e.g. \cite[6.1.13]{zbMATH01231230}. The normalisation used therein is however different from ours: the Hermite polynomials are normalised so that they are orthogonal in $L^2(e^{-x^2} \d x)$ instead of $L^2(e^{-x^2/2} \d x)$.
    When $x=y$ the left hand side sum contains only nonnegative terms. Each of these terms is therefore bounded by the right hand side which gives \eqref{E:L_Hermite_upper_bound}.
\end{proof}

As a matter of comparison, when $k$ is even, $|H_k(0)| = 2^{-k/2} k!/(k/2)!$ which behaves, up to polynomial factor, like $\sqrt{k!}$ when $k \to \infty$. So the right hand side of \eqref{E:L_Hermite_upper_bound} is a fairly good bound when $x$ is fixed and $\rho$ is close to 1. On the other hand, if $x$ grows with $k$, one might want to take different values of $\rho$ to get better bounds.

\smallskip

We now describe some of the interactions between Hermite polynomials and Gaussian random variables.

\begin{lemma}\label{L:Hermite_martingale1d}
    Let $(B_t)_{t \ge 0}$ be a standard one-dimensional Brownian motion. For all $k \ge 0$, the process $(t^{k/2} H_k (t^{-1/2}B_t))_{t \ge 0}$ is a martingale.
\end{lemma}

\begin{proof}
    This is a well-known fact; this corresponds for instance to the special case $\theta=1/2$ of \cite[Lemmas 9.7 and 9.9]{JLQ}. A short proof of this result goes as follows.
    By \eqref{E:Hermite_generating_function}, for all $\gamma \in \R$,
    \[
    e^{\gamma B_t - \gamma^2 t/2} = \sum_{k\ge 0} \frac{\gamma^k}{k!} t^{k/2} H_k(B_t / \sqrt{t}), \quad t \ge 0.
    \]
    Since the left hand side is a martingale for all $\gamma$, each of the coefficients of the expansion has to be a martingale as well, proving Lemma \ref{L:Hermite_martingale1d}.
\end{proof}

\begin{lemma}\label{L:Hermite_martingale}
Let $X$ be a standard Gaussian and $(X_1,X_2)$ be a centred Gaussian vector with $\E[X_1^2]=\E[X_2^2]=1$ and $\E[X_1X_2]=\rho \in [-1,1]$.
Let $\sigma,\sigma_1,\sigma_2 \in (-1,1)$, $m,m_1,m_2 \in \R$ and $k \ge 1$. Then
\begin{equation}
\label{E:Hermite1}
\E[H_k(\sigma X + m)] = (1-\sigma^2)^{k/2} H_k(m/\sqrt{1-\sigma^2})
\end{equation}
and
\begin{align}\label{E:Hermite2}
& \E[H_k(\sigma_1 X_1 + m_1) H_k(\sigma_2 X_2 + m_2)] \\
& = \sum_{l=0}^k (1-\sigma_1^2)^{\frac{k-l}{2}} (1-\sigma_2^2)^{\frac{k-l}{2}} \rho^l \sigma_1^l\sigma_2^l \frac{k!^2}{l!(k-l)!^2} H_{k-l}\big( \frac{m_1}{\sqrt{1-\sigma_1^2}} \big) H_{k-l}\big( \frac{m_2}{\sqrt{1-\sigma_2^2}} \big).
\nonumber
\end{align}
\end{lemma}

\begin{proof}
One could observe that \eqref{E:Hermite1} is a rephrasing of the martingale property of $(t^{k/2} H_k (t^{-1/2}B_t))_{t \ge 0}$.
Alternatively, by the orthogonality property of the Hermite polynomials, $\E[H_k(X)] = 1$ if $k=0$ and 0 otherwise. Hence, \eqref{E:Hermite1} follows by applying \eqref{E:L_Hermite_umbral1} to $\rho = \sigma$, $u=X$ and $v=m/\sqrt{1-\sigma^2}$.

We now move on to \eqref{E:Hermite2} and consider first the case $X_1=X_2$ ($\rho=1$). For all $k,l \ge 0$, $\E[H_k(X) H_l(X)] = k! \mathbf{1}_{k=l}$ (see e.g. \cite[6.1.5]{zbMATH01231230}). Hence, by \eqref{E:L_Hermite_umbral1},
\begin{align*}
    & \E[H_k(\sigma_1 X + m_1) H_k(\sigma_2 X + m_2)] \\
    & = \sum_{l=0}^k (1-\sigma_1^2)^{\frac{k-l}{2}} (1-\sigma_2^2)^{\frac{k-l}{2}} \sigma_1^l\sigma_2^l \frac{k!^2}{l!(k-l)!^2} H_{k-l}\big( \frac{m_1}{\sqrt{1-\sigma_1^2}} \big) H_{k-l}\big( \frac{m_2}{\sqrt{1-\sigma_2^2}} \big).
\end{align*}
Consider now the general case $\rho \in [-1.1]$.
We have $(X_1,X_2) \overset{\mathrm{(d)}}{=} (X_1, \rho X_1 + \sqrt{1-\rho^2} Z)$ where $Z$ is a standard Gaussian independent of $X_1$. So by \eqref{E:Hermite1} we have
\begin{align*}
& \E[H_k(\sigma_1 X_1 + m_1) H_k(\sigma_2 X_2 + m_2)] \\
& = (1-(1-\rho^2)\sigma_2^2)^{k/2}
\E[ H_k(\sigma_1 X_1 + m_1) H_k((\rho \sigma_2 X_1 + m_2)/\sqrt{1-(1-\rho^2)\sigma_2^2})].
\end{align*}
We are now back to the case $X_1=X_2$ with $(\sigma_1,m_1,\sigma_2,m_2) \leftarrow (\sigma_1,m_1,\rho\sigma_2/\sqrt{1-\sigma_2^2},m_2/\sqrt{1-\sigma_2^2})$. Rearranging, we obtain \eqref{E:Hermite2} as desired.
\end{proof}

\subsection{Proof of Theorem \ref{T:main} when \texorpdfstring{$\gamma=0$}{gamma=0}}\label{SS:gamma=0}

In this short section, we prove Theorem \ref{T:main} when $\gamma=0$. Although this is well known (see e.g. \cite[Section 2.3]{zbMATH06732240}), we choose to present the argument for completeness and because we need to treat this case separately. The proof is an immediate consequence of \eqref{E:Hermite2} above.

Let $X$ be a log-correlated Gaussian field defined on some domain $D\subset\R^d$ and $(X_\eps)_{\eps\in(0,1)}$ be its mollified version as above Theorem \ref{T:main}. Let $f:D \to \R$ be a bounded measurable function compactly supported in $D$. By \eqref{E:Hermite2} applied to $m_1=m_2=0$, $\sigma_1=\sigma_2=0$ and $\rho = \E[X_\eps(x)X_\delta(y)]/\sqrt{\var(X_\eps(x))\var(X_\delta(y))}$, we have for all $\eps,\delta \in (0,1)$,
\begin{align*}
    \E\Big[\Big( \int_D \wick{X_\eps(x)^k} f(x) \d x\Big)\Big( \int_D \wick{X_\delta(x)^k} f(x) \d x\Big)\Big]
    & = \int_{D\times D} \E[\wick{X_\eps(x)^k}\wick{X_\delta(y)^k}] f(x)f(y) \d x \d y\\
    & = k! \int_{D\times D} \E[X_\eps(x)X_\delta(y)]^k f(x)f(y) \d x \d y.
\end{align*}
By dominated convergence theorem and the pointwise convergence of $\E[X_\eps(x)X_\delta(y)]$ to $C(x,y)$, the above implies that $(\int_D \wick{X_\eps(x)^k} f(x) \d x)_{\eps\in(0,1)}$ is Cauchy in $L^2$. This proves \eqref{E:T_main1} when $\gamma=0$. 
It also shows that
\begin{align*}
    \sup_{\eps \in (0,1)}\E\Big[\Big( \int_D \wick{X_\eps(x)^k} f(x) \d x\Big)^2\Big]
    & \le k! \|f\|_\infty^2 \int_{\mathrm{supp}(f)^2} \Big(\log_+\frac1{|x-y|}+C\Big)^k \d x \d y\\
    & \le C(f,d) k! \int_0^1 r^{d-1} |\log r|^k \d r
    = C(f,d) k!^2 d^{-k-1}
\end{align*}
where the last equality follows from successive integration by parts. In particular, if $|\gamma'|<\sqrt{d}$,
\[
\sum_{k \ge 0} \frac{|\gamma'|^k}{k!} \sup_{\eps \in (0,1)} \E\Big[\Big| \int_D \wick{X_\eps(x)^k} f(x) \d x\Big|\Big] < \infty.
\]
The $L^1$ convergence of \eqref{E:T_main3} when $|\gamma'|<\sqrt{d}$ then follows. A simple extension of the above shows that the convergence also holds in $L^2$. This is however restricted to the case $\gamma=0$ since, for any other value of $\gamma$, some part of the disc of convergence $\{\gamma'\in\C:|\gamma'-\gamma|<\sqrt{d}-|\gamma|/\sqrt2\}$ lies outside of the $L^2$-phase $\{\gamma'\in\C:|\gamma'|<\sqrt{d}\}$.
\qed

\subsection{Basic properties of almost \texorpdfstring{$*$}{star}-scale invariant fields}

This section gathers a few properties of almost $*$-scale invariant fields.

\begin{lemma}[Scaling covariance 1]\label{L:scaling_covariance}
For $r_0 = e^{-t_0} \in (0,1]$,
\begin{align}\label{E:scaling_covariance_star_scale}
    (X_t^\af(r_0 \cdot),X_{t_0}^\af) & \overset{\mathrm{(d)}}{=} (X_{t - t_0}^{\af r_0^\alpha}(\cdot) + X_{t_0}^\af(r_0\cdot), X_{t_0}^\af), \qquad t \in [t_0,\infty],\\
    \text{and} \qquad (X_\eps^\af(r_0 \cdot),X_{t_0}^\af) & \overset{\mathrm{(d)}}{=} (X_{\eps/r_0}^{\af r_0^\alpha}(\cdot) + X_{t_0,\eps}^\af(r_0 \cdot), X_{t_0}^\af), \qquad \eps \in (0,r_0],
    \label{E:scaling_covariance}
\end{align}    
where, on the right hand side of \eqref{E:scaling_covariance_star_scale}, respectively \eqref{E:scaling_covariance}, the fields $X_{t - t_0}^{\af r_0^\alpha}$ and $X_{t_0}^\af$ are independent,
respectively the fields $X_{\eps/r_0}^{\af r_0^\alpha}$ and $X_{t_0}^\af$ are independent and $X_{t_0,\eps}^\af = X_{t_0}^\af * \varphi_\eps$.
\end{lemma}

\begin{proof}
    Let $t_0 = -\log r_0 \ge 0$ and $t \in [t_0,\infty]$.
    By definition \eqref{E:X_WN}, $X^\af_t - X^\af_{t_0}$ and $X^\af_{t_0}$ correspond to white noise integrals on disjoint space-time subsets. These fields are thus independent. Moreover, doing the change of variables $v = u-t_0$ and $w = r_0^{-1} z$, we have
    \begin{align*}
        & (X^\af_t - X^\af_{t_0})(r_0 x) = \int_{\R^d \times (t_0,t)} e^{du/2} \tilde\kf(e^u(z-r_0x)) \sqrt{1-\af e^{-\alpha u}} \d W(z,u)\\
        & = \int_{\R^d \times (0,t-t_0)} e^{dv/2} \tilde\kf(e^v(w-x)) \sqrt{1-\af r_0^\alpha e^{-\alpha v}} r_0^{-d/2} \d W(r_0 w,v+t_0) \\
        & \overset{(d)}{=} \int_{\R^d \times (0,t-t_0)} e^{dv/2} \tilde\kf(e^v(w-x)) \sqrt{1-\af r_0^\alpha e^{-\alpha v}} \d W(w,v)
        \overset{(d)}{=} X^{\af r_0^\alpha}_{t-t_0}(x).
    \end{align*}
    This concludes the proof of \eqref{E:scaling_covariance_star_scale}. The proof of \eqref{E:scaling_covariance} then follows by integrating \eqref{E:scaling_covariance_star_scale} with $t=\infty$ against the mollifier $\varphi_\eps$.
\end{proof}

\begin{lemma}[Scaling covariance 2]\label{L:scaling_covariance_Hermite}
    Let $r_0 = e^{-t_0} \in (0,1]$, $\eps \in (0,r_0]$, $y \in \R^d$ and $k \ge 0$. 
    Then $(: X_\eps^\af(y+r_0 \cdot)^k e^{\gamma X_\eps^\af(y+r_0 \cdot)}:, X_{t_0}^\af(y+\cdot))$ has the same distribution as
    \begin{equation}
    \Big( \sum_{j=0}^k \binom{k}{j} :(X_{\eps/r_0}^{\af r_0^\alpha})^j e^{\gamma X_{\eps/r_0}^{\af r_0^\alpha}}: \times
    :X_{t_0,\eps}^\af(r_0\cdot)^{k-j} e^{\gamma X_{t_0,\eps}^\af(r_0\cdot)}: , X_{t_0}^\af \Big)
    \end{equation}
    where the fields $X_{\eps/r_0}^{\af r_0^\alpha}$ and $X_{t_0}^\af$ are independent and $X_{t_0,\eps}^\af = X_{t_0}^\af *\varphi_\eps$.
\end{lemma}

\begin{proof}
To ease notations, we only consider the first marginal.
    By \eqref{E:scaling_covariance} and translation invariance,
$
X_\eps^\af(y+r_0\cdot) \overset{(d)}{=} X_{\eps/r_0}^{\af r^\alpha} + X_{t_0,\eps}^\af(r_0\cdot),
$
where the joint law is as in the statement.
By \eqref{E:L_Hermite_umbral1} applied to
\[
u = \frac{X_{\eps/r_0}^{\af r_0^\alpha} - \gamma (\sigma_{\eps/r_0}^{\af r_0^\alpha})^2}{\sigma_{\eps/r_0}^{\af r_0^\alpha}},
\quad v = \frac{X_{t_0,\eps}^\af(r_0\cdot) - \gamma (\sigma_{t_0,\eps}^\af)^2}{\sigma_{t_0,\eps}^\af}
\quad \text{and} \quad \rho = \frac{\sigma_{\eps/r_0}^{\af r^\alpha}}{\sigma_{\eps}^\af},
\]
we have
\begin{align*}
& (\sigma_\eps^a)^k H_k \Big( \frac{X_\eps^\af(y+r_0\cdot) - \gamma (\sigma_\eps^a)^2}{\sigma_\eps^a} \Big)
\overset{\mathrm{(d)}}{=} (\sigma_\eps^a)^k H_k \Big( \frac{X_{\eps/r_0}^{\af r^\alpha} + X_{t_0,\eps}^\af(r_0\cdot) - \gamma (\sigma_\eps^a)^2}{\sigma_\eps^a} \Big)
\\
& = \sum_{j=0}^k \binom{k}{j} (\sigma_{\eps/r_0}^{\af r_0^\alpha})^j H_j \Big( \frac{X_{\eps/r_0}^{\af r_0^\alpha} - \gamma (\sigma_{\eps/r_0}^{\af r_0^\alpha})^2}{\sigma_{\eps/r_0}^{\af r_0^\alpha}} \Big)  (\sigma_{t_0,\eps}^a)^{k-j} H_{k-j} \Big( \frac{X_{t_0,\eps}^\af(r_0\cdot) - \gamma (\sigma_{t_0,\eps}^\af)^2}{\sigma_{t_0,\eps}^\af} \Big)
\end{align*}
which is equivalent to the stated claim.
\end{proof}

\begin{lemma}\label{L:covariance}
    Writing $\log_+(\cdot) = \max(\log (\cdot), 0)$, we have for all $x,y 
    \in [0,1]^d$,
    \begin{gather}\label{E:L_covariance1}
        C_{\eps,\delta}(x,y) = \log_+ \frac{1}{\eps \vee \delta \vee |x-y|} + O(1), \\
        \label{E:L_covariance2}
        C_{s,t}(x,y) = \log_+ \frac1{e^{-s} \vee e^{-t} \vee |x-y|} + O(1), \\
        \label{E:L_covariance3}
        C_{t,\eps}(x,y) = \log_+ \frac{1}{\eps \vee e^{-t} \vee |x-y|} + O(1), \\
        \label{E:L_covariance5}
        \sup_{\af \in [0,1]} \sup_{t \ge 0} \sup_{\eps \in [0,1]} |(\sigma_{t,\eps}^\af)^2 - t \wedge |\log \eps| | < \infty, \\
        \label{E:L_covariance6}
        \sup_{\af \in [0,1]} \sup_{t \ge 0} \sup_{\eps \in [0,1]} \sup_{x \in Q} |\E[X_t^\af(0) X_{t,\eps}^\af( (e^{-t} \vee \eps) x)] - t \wedge |\log \eps| | < \infty.
    \end{gather}
    Moreover, for all $x,y \in [0,1]^d$,
    \begin{equation}\label{E:L_covariance7}
        \lim_{\eps_1,\eps_2\to 0} C_{\eps_1,\eps_2}(x,y) = C(x,y).
    \end{equation}
\end{lemma}

\begin{proof}
    All these estimates are very standard; see e.g. \cite[Proposition 4.1]{zbMATH07172346}. To see that they are uniform in $\af \in [0,1]$, one can always upper bound (resp. lower bound) the covariances by the corresponding ones with $\af = 0$ (resp. $\af=1$).
\end{proof}

We define the following filtrations in terms of the white noise $W$ defining $X$ (see \eqref{E:X_WN}):
for all $x,y \in \R^d$, $s,t \ge 0$ and $\eps \in [0,1]$,
\begin{equation}
\label{E:filtration2}
\Fc_t(x,\eps) = \sigma(W(z,u): z \in B(x,\eps+e^{-u}), u \in (0,t))
\quad \text{and} \quad
\Fc_{s,t}(x,y,\eps) = \sigma(\Fc_s(x,\eps) \cup \Fc_t(y,\eps)).
\end{equation}
By construction, $X_t(z)$ is $\Fc_t(x,\eps)$-measurable for all $z \in B(x,\eps)$.

\begin{lemma}\label{L:independence}
Let $\eps \in (0,1]$ and $s,t \ge 0$. For all $x,y \in \R^d$ such that $|x-y| \ge 2\eps + e^{-s}+e^{-t}$, the random variables $X_\eps(x) - \E[X_\eps(x)\vert \Fc_s(x,\eps)]$ and $X_\eps(y) - \E[X_\eps(y)\vert \Fc_t(y,\eps)]$ are independent and independent of $\Fc_{s,t}(x,y,\eps)$.
\end{lemma}

\begin{proof}
The condition implies that the balls $B(x,\eps+e^{-s})$ and $B(y,\eps+e^{-t})$ are disjoint.
\end{proof}

\subsection{Two elementary lemmas}

We record here two elementary results for ease of future reference.

\begin{lemma}\label{L:elementary_sum}
    Let $c_1>0$. For any $\eta>0$, there exists $C=C(\eta,c_1)>0$ such that for all $k \ge 1$ and $\mathfrak{m} \ge 1$,
    \begin{equation}
        \sum_{n \ge \mathfrak{m}} e^{-c_1 n} n^k \le C k! \Big( \frac{1+\eta}{c_1} \Big)^k e^{-\frac{\eta c_1}{1+\eta} \mathfrak{m}}.
    \end{equation}
\end{lemma}

\begin{proof}
    We bound
    \[
    n^k = \Big( \frac{1+\eta}{c_1} \Big)^k \Big( \frac{nc_1}{1+\eta} \Big)^k \le \Big( \frac{1+\eta}{c_1} \Big)^k k! e^{nc_1/(1+\eta)}
    \]
    to get that
    \[
    \sum_{n \ge \mathfrak{m}} e^{-c_1 n} n^k
    \le \Big( \frac{1+\eta}{c_1} \Big)^k k! \sum_{n \ge \mathfrak{m}} e^{-\frac{\eta c_1}{1+\eta} n} = \Big( \sum_{n \ge 0} e^{-\frac{\eta c_1}{1+\eta} n}\Big) \Big( \frac{1+\eta}{c_1} \Big)^k k! e^{-\frac{\eta c_1}{1+\eta} \mathfrak{m}}
    \]
    which is the desired estimate.
\end{proof}

\begin{lemma}\label{L:very_elementary}
Let $C_1 \ge0$, $C_2 \in [0,1)$ and $(M_N)_{N \ge 0}$ be a sequence of numbers in $[0,\infty]$ satisfying
\[
M_N \le \max(M_{N-1}, C_1 + C_2 M_{N-1}), \qquad N \ge 1.
\]
Then $\sup_{N \ge 0} M_N \le \max(M_0, C_1/(1-C_2))$.
\end{lemma}

\begin{proof}
If $M_0 \ge C_1/(1-C_2)$, then $M_1 \le \max(M_0,C_1+C_2M_0) = M_0$.
If $M_0 \le C_1/(1-C_2)$, then $M_1 \le \max(M_0,C_1+C_2M_0) = C_1+C_2M_0 \le C_1 + C_2 C_1/(1-C_2) = C_1/(1-C_2)$. Overall, $M_1 \le \max(M_0, C_1/(1-C_2))$. The same bound can be obtained for $M_N$ for any $N$ by induction.
\end{proof}

\section{Proof of Theorem \ref{T:convergence_k}}\label{S:convergence}

This section is dedicated to the proof of Theorem \ref{T:convergence_k}. By symmetry, we can assume that $\gamma \ge 0$ and since the case $\gamma=0$ was treated in Section~\ref{SS:gamma=0}, we will assume that $\gamma>0$ throughout the section. We first introduce good events and state three propositions, Proposition \ref{P:bounded_L2}, \ref{P:bound_pmoment} and \ref{P:bad_L1} below. We then show how Theorem \ref{T:convergence_k} follows from them. 
In Sections \ref{SS:2nd}, \ref{SS:higher} and \ref{SS:1st}, we will prove Propositions \ref{P:bounded_L2}, \ref{P:bound_pmoment} and \ref{P:bad_L1} respectively.
To ease notations, we will actually only prove Theorem \ref{T:convergence_k} for $\eps \in \{e^{-n}\}_{n \ge 0}$ instead of considering any $\eps \in (0,1)$. In this whole section, we will thus always restrict to such ``exponential scales'' without necessarily writing it explicitly. Trivial adaptations deal with the general case. 

\begin{notation}[Lattice, cube]\label{N:lattice}
Let $n \ge 0$. Define $\Lambda_n := e^{-n} \Z^d$ and for $x \in \Lambda_n$, let $Q_n(x) := x + \frac12 e^{-n} [-1,1)^d$. For any $x \in \R^d$, we denote by $x_n$ the unique element of $\Lambda_n$ such that $x \in Q_n(x_n)$.    
Let $Q$ be the cube $[0,1]^d$.
\end{notation}

\noindent\textbf{Good events.}
Let $\hat\gamma > \gamma$ be a parameter.
Recalling that $X_n^\af$ stands for the $*$-scale regularisation \eqref{E:X_approx_star_scale}, we define the good event
\begin{align}
E_n(x) := \{ X_n^\af(x) \le \hat\gamma n \}, \qquad n \ge 1, \quad x \in \Lambda_n,\\
G_{\eps,\delta}(x) := \bigcap_{n=n_0}^N E_n(x_n), \qquad \delta = e^{-n_0} > \eps = e^{-N}, \quad x \in \R^d.
\end{align}
The parameter $\hat\gamma$ will always be the one we use in the definition of these events and so we omit it from our notations.
Using good events which are defined in terms of the $*$-scale regularisation of the field will lead to simplifications but is not crucial.
Let $\delta,\eps \in \{e^{-n}\}_{n \ge 0}$ with $\delta > \eps$, $f:\R^d\to\R$ be a bounded measurable function with compact support and $k\ge0$.
Recall from \eqref{E:defI}, that we denote by $I_\eps^\af(f,k)$ the integral of $\wick{X_\eps^k e^{\gamma X_\eps}}$ against $f$.
We can split
\begin{align}\label{E:split_good_bad}
\hspace{60pt} I_\eps^\af(f,k) & = I_{\eps,\delta}^{\af,\good}(f,k) + I_{\eps,\delta}^{\af,\bad}(f,k)
\hspace{50pt}\text{where}\\
I_{\eps,\delta}^{\af,\good}(f,k) & := \int_{\R^d} f(x) : X_\eps^\af(x)^k e^{\gamma X_\eps^\af(x)} : \mathbf{1}_{G_{\eps,\delta}(x)} \d x, \\
I_{\eps,\delta}^{\af,\bad}(f,k) & := \int_{\R^d} f(x) : X_\eps^\af(x)^k e^{\gamma X_\eps^\af(x)} : \mathbf{1}_{G_{\eps,\delta}(x)^c} \d x.
\end{align}

\begin{proposition}[Second moment, good event]\label{P:bounded_L2}
Assume that $\hat\gamma \in (\gamma,\sqrt{2d}\wedge(2\gamma))$.
Let $k \ge 0$, $\af \in [0,1]$, $f: [0,1]^d \to \R$ be a bounded function and $\delta \in (0,1)$. The sequence $(I_{\eps,\delta}^{\af,\good}(f,k))_{\eps \in (0,\delta)}$ is a Cauchy sequence in $L^2$. Moreover, for all $u > \sqrt2/(\hat\gamma-\gamma)$, there exists $C>0$ which may depend on $\gamma,\hat\gamma$ and $u$ but not on $k$, $\eps$, $\delta$, $\af$ or $f$ such that for all $\eps \in (0,\delta)$,
\begin{equation}\label{E:P_bdd_L2}
\E[ I_{\eps,\delta}^{\af,\good}(f,k)^2 ] \le C \|f\|_\infty^2 u^{2k} k!^2 \delta^{-d}.
\end{equation}
\end{proposition}

\begin{remark}[Restriction $\hat\gamma<2\gamma$]\label{R:restriction}
We now comment on the assumption that $\hat\gamma < 2\gamma$ made in the above statement and then discuss its consequences.
For the sake of this discussion, we can focus on the classical case $k=0$.
When computing the second moment of $I_{\eps,\delta}^{\af,\good}(f,0)$, we will encounter a term of the form
\[
\E[\wick{e^{\gamma X_\eps(x)}}\wick{e^{\gamma X_\eps(y)}} \mathbf{1}_{G_{\eps,\delta}(x)\cap G_{\eps,\delta}(y)}],
\]
where $x,y \in [0,1]^d$. If these points are such that $\eps<|x-y|<\delta$, the above intersection of good events contains the event that
$X_n(x) \le \hat\gamma n$
where $n$ is such that $e^{-n}$ is of order $|x-y|$. Hence, by applying Cameron--Martin formula, we get that the above display is at most
\begin{align*}
    e^{\gamma^2 \E[X_\eps(x)X_\eps(y)]} \P(X_n(x) + \gamma\E[X_n(x)(X_\eps(x)+X_\eps(y))] \le \hat\gamma n).
\end{align*}
The above choice of scales guarantees that $\E[X_n(x)(X_\eps(x)+X_\eps(y))] = 2n +O(1)$ and so we are dealing with a probability of the form
\[
\P(X_n(x) \le (\hat\gamma-2\gamma)n + O(1)).
\]
If $\hat\gamma < 2\gamma$, this probability is exponentially small in $n$. However, if $\hat\gamma \ge 2\gamma$ this probability is not particularly small which eventually leads to poorer bounds, especially in terms of the growth as $k\to \infty$.
\end{remark}

The restriction $\hat\gamma<2\gamma$ is problematic because we want to take $\hat\gamma$ arbitrarily close to $\sqrt{2d}$ to make 
$\sqrt{2}/(\hat{\gamma}-\gamma)$ arbitrarily close to $\sqrt{2}/(\sqrt{2d}-\gamma)$ to achieve the desired estimate \eqref{E:T_convergence_k}. When $\gamma<\sqrt{d/2}$, $2\gamma<\sqrt{2d}$ and the maximum value of $\hat\gamma$ is not the desired one.
We resolve this problem by considering higher moments which will effectively replace the restriction $\hat\gamma<2\gamma$ by $\hat\gamma<p\gamma$ where $p$ is the order of the moment:

\begin{proposition}[Higher moments, good event]\label{P:bound_pmoment}
    Let $p\ge 3$ be an integer. Assume that $\gamma < \sqrt{2d/p}$ and $\hat\gamma\in(\gamma, \sqrt{2d}\wedge(p\gamma))$. For all $u > \sqrt{2}/(\hat\gamma-\gamma)$, there exists $C>0$ which may depend on $\gamma$, $\hat\gamma$, $u$ and $p$ such that for all $\af\in[0,1]$, $k\ge 0$, $f:[0,1]^d\to\R$ bounded measurable function, $\delta\in(0,1)$, $\eps\in(0,\delta)$,
    \begin{equation}\label{E:P_bdd_Lp}
|\E[ I_{\eps,\delta}^{\af,\good}(f,k)^{p} ]| \le C \|f\|_\infty^{p} u^{pk} k!^{p} \delta^{-dp/2}.
\end{equation}
\end{proposition}

Finally, the next proposition states that the contribution of the bad event is small in $L^1$:

\begin{proposition}[First moment, bad event]\label{P:bad_L1}
    Assume that $\hat\gamma \in (\gamma,\sqrt{2d})$. For any $u > \sqrt2/(\hat\gamma-\gamma)$,
    \begin{equation}\label{E:P_bdd_L1}
        \lim_{\delta \to 0} \sup_{\eps \in (0,\delta)} \sup_{k \ge 0} \frac{1}{u^k k!} \sup_{\af \in [0,1]} \sup_{0<\|f\|_\infty<\infty} \frac1{\|f\|_\infty} \E[ |I_{\eps,\delta}^{\af,\bad}(f,k)| ] = 0,
    \end{equation}
    where the last supremum ranges over all bounded measurable functions $f:[0,1]^d \to \R$.
\end{proposition}

\begin{proof}[Proof of Theorem \ref{T:convergence_k}, assuming Propositions \ref{P:bounded_L2}, \ref{P:bound_pmoment} and \ref{P:bad_L1}]
    Let $f: \R^d \to \R$ be a bounded measurable function with compact support. Without loss of generality, we may assume that $f$ is supported in $[0,1]^d$.
    Combining Propositions \ref{P:bounded_L2} and \ref{P:bad_L1}, we immediately obtain that $(I_\eps^\af(f,k))_{\eps \in (0,1)}$ is a Cauchy sequence in $L^1$ and thus converges. It remains to prove the uniform estimate \eqref{E:T_convergence_k}. Let $\eta >0$. Let $\hat\gamma \in (\gamma,\sqrt{2d})$ close enough to $\sqrt{2d}$ so that $(1+\eta)\frac{\sqrt2}{\sqrt{2d}-\gamma} > \frac{\sqrt2}{\hat\gamma-\gamma}$. Let $u$ be such that $(1+\eta)\frac{\sqrt2}{\sqrt{2d}-\gamma} >u> \frac{\sqrt2}{\hat\gamma-\gamma}$. Let $\delta>0$. From the splitting \eqref{E:split_good_bad} in terms of good and bad events and then by Hölder's inequality, we can bound for all $\eps \in (0,\delta)$,
    \begin{align*}
        \E[|I_\eps^\af(f,k)|] \le \E[|I_{\eps,\delta}^{\af,\good}(f,k)|] + \E[|I_{\eps,\delta}^{\af,\bad}(f,k)|]
        \le \E[|I_{\eps,\delta}^{\af,\good}(f,k)|^{2p}]^{1/(2p)} + \E[|I_{\eps,\delta}^{\af,\bad}(f,k)|],
    \end{align*}
    where $p=p(\gamma) \ge 1$ is the smallest integer such that $2p\gamma \ge \sqrt{2d}$. If $\gamma\ge\sqrt{d}/2$, then $p=1$ and we apply Proposition \ref{P:bounded_L2} to bound the first right hand side term. Otherwise, $p\ge 2$ and because $2(p-1)\gamma < \sqrt{2d}$, we must have $\gamma < \sqrt{2d}/(2(p-1)) \le \sqrt{2d}/\sqrt{2p}$ and we can then apply Proposition \ref{P:bound_pmoment}. The second right hand side term is handled by Proposition \ref{P:bad_L1}. Overall, the right hand side is bounded by $C \|f\|_\infty k! u^k$ where $C>0$ is independent of $k$, $f$ and $\eps$. Because $u$ is smaller than $(1+\eta)\frac{\sqrt2}{\sqrt{2d}-\gamma}$, this proves the desired estimate \eqref{E:T_convergence_k}.
\end{proof}

\subsection{Second moment, good event}\label{SS:2nd}

Let $\eps,\eps_1,\eps_2,\delta \in \{e^{-n}\}_{n \ge 0}$ with $\eps,\eps_1,\eps_2 \le \delta$.
In this proof, we always write $\eps =e^{-N}$, $\eps_1 =e^{-N_1}$, $\eps_2 =e^{-N_2}$ and $\delta=e^{-n_0}$.
Exchanging expectation and integral and then using Cameron--Martin theorem, we get that
\begin{equation}\label{E:linkAI}
\E[ I_{\eps_1,\delta}^{\af,\good}(f,k)I_{\eps_2,\delta}^{\af,\good}(f,k) ]
= \int_{[0,1]^d \times [0,1]^d} f(x)f(y) e^{\gamma^2 C_{\eps_1,\eps_2}(x,y)} A(k,x,y,\eps_1,\eps_2) \d x \d y
\end{equation}
where
\begin{align}
A(k,x,y,\eps_1,\eps_2) := \sigma_{\eps_1}^k \sigma_{\eps_2}^k \E\Big[ H_k\Big( \frac{X_{\eps_1}(x)+\gamma C_{\eps_1,\eps_2}(x,y)}{\sigma_{\eps_1}} \Big) H_k\Big( \frac{X_{\eps_2}(y)+\gamma C_{\eps_1,\eps_2}(x,y)}{\sigma_{\eps_2}} \Big) \mathbf{1}_{G_{N_1,N_2}(x,y)} \Big]\label{E:def_A}
\end{align}
and where $G_{N_1,N_2}(x,y)$ is the event $G_{\eps_1,\delta}(x)\cap G_{\eps_2,\delta}(y)$ twisted by the Cameron--Martin shift. Specifically,
\begin{equation}
    \label{E:goodgood}
G_{N_1,N_2}(x,y) = \bigcap_{n=n_0}^{N_1} E_n^{(1)}(x,y) \cap \bigcap_{n=n_0}^{N_2} E_n^{(2)}(x,y),
\end{equation}
where for all $n \in \{n_0,\dots,N\}$,
\begin{align}\nonumber
E_n^{(1)}(x,y) & = 
 \{ X_n(x_n) \le \hat\gamma n - \gamma C_{n,\eps_1}(x_n,x) - \gamma C_{n,\eps_2}(x_n,y) \} \\
E_n^{(2)}(x,y) & = \{ X_n(y_n) \le \hat\gamma n - \gamma C_{n,\eps_2}(y_n,y) - \gamma C_{n,\eps_1}(y_n,x) \}.
\end{align}
The crux of the proof of Proposition \ref{P:bounded_L2} resides in bounding $A(k,x,y,\eps_1,\eps_2)$ and showing that it converges pointwise as $\eps_1,\eps_2\to 0$. We encapsulate this in the following lemma.

\begin{lemma}\label{L:boundA}
For any $u > \sqrt2/(\hat{\gamma}-\gamma)$, there exists $C >0$ (which may depend on $u,\gamma,\hat\gamma$) such that for all $k \ge 0$, $\eps_1,\eps_2 \in (0,1)$, $\delta \in (\eps,1)$ and $x,y \in [0,1]^d$, letting $r = \max(\eps_1,\eps_2,\min(\delta,|x-y|))$,
\begin{equation}
    \label{E:boundA}
|A(k,x,y,\eps_1,\eps_2)| \le C u^{2k} k!^2 r^{\frac12(2\gamma-\hat\gamma)_+^2 -(\hat\gamma-\gamma)^2}.
\end{equation}
In the above display, $(2\gamma-\hat\gamma)_+ = \max(2\gamma-\hat\gamma,0)$.
Moreover, for all disjoint $x,y \in [0,1]^d$ and $k \ge0$, $A(k,x,y,\eps_1,\eps_2)$ converges as $\eps_1,\eps_2\to 0$.
\end{lemma}

Bounding this expectation requires care. For instance, putting absolute values around the Hermite polynomials will make it diverge. We will use a ``martingale trick'' to integrate out some randomness before putting such absolute values around the Hermite polynomials. 
We illustrate this idea with a toy model, corresponding to a one-point version of the above.
Related ideas can be found in \cite[Section 5]{zbMATH07493822} in the context of complex GMC.

\smallskip

\noindent\textbf{Martingale trick.} 
Let $(B_t)_{t \ge 0}$ be a one-dimensional standard Brownian motion and consider the events
\begin{equation}
\label{E:MG_trick}
E_T = \{ \forall t \in \{0,\dots, T\}: B_t \le t\}, \qquad T \in \N.
\end{equation}
Letting $k \ge 0$, we wish to show that $\sup_{T \ge 1} T^{k/2} |\E[H_k(B_T/\sqrt{T}) \mathbf{1}_{E_T}]| < \infty$. Let $T \ge 1$. We can rewrite
\begin{equation}\label{E:MG_trick1}
T^{k/2} \E[H_k(B_T/\sqrt{T}) \mathbf{1}_{E_T}] = T^{k/2} \E[H_k(B_T/\sqrt{T}) \mathbf{1}_{E_{T-1}}] - T^{k/2} \E[H_k(B_T/\sqrt{T}) \mathbf{1}_{E_{T-1}\cap\{B_T > T\}}].
\end{equation}
A standard computation with the Gaussian distribution shows that the absolute value of the second right hand side term is at most
\[
T^{k/2} \E[|H_k(B_T/\sqrt{T})| \mathbf{1}_{\{B_T > T\}}] \le C(k) T^{k-1/2} e^{-T/2}.
\]
On the other hand, because $(t^{k/2} H_k(B_t/\sqrt{t}))_{t \ge 0}$ is a martingale (Lemma \ref{L:Hermite_martingale1d}), the first term on the right hand side of \eqref{E:MG_trick1} equals $(T-1)^{k/2} \E[H_k(B_{T-1}/\sqrt{T-1}) \mathbf{1}_{E_{T-1}}]$. Hence
\[
| T^{k/2} \E[H_k(B_T/\sqrt{T}) \mathbf{1}_{E_T}] - (T-1)^{k/2} \E[ H_k(B_{T-1}/\sqrt{T-1}) \mathbf{1}_{E_{T-1}}] | \le C(k) T^{k-1/2} e^{-T/2}.
\]
Iterating this argument, we deduce that
\[
| T^{k/2} \E[H_k(B_T/\sqrt{T}) \mathbf{1}_{E_T}] | \le C(k) \sum_{t=1}^T t^{k-1/2} e^{-t/2} \le C(k) \sum_{t=1}^\infty t^{k-1/2} e^{-t/2} < \infty.
\]

\begin{proof}[Proof of Lemma \ref{L:boundA}]
To ease notations, let us prove the lemma when $\eps_1=\eps_2$.
Let $\delta = e^{-n_0}$ and $\eps = e^{-N}$ where $N \ge n_0 \ge 0$.

\noindent\textbf{Claim.} There exists $C>0$ such that for all $|x-y| < e^{-1} \delta$,
\begin{equation}\label{E:claim_proba}
\P(G_{\eps,\delta}(x,y)) \le C  (\eps \vee |x-y|)^{(2\gamma-\hat\gamma)_+^2/2}.
\end{equation}
Indeed, if $|x-y| \le \eps$, $C_{N,\eps}(x_N,x)$ and $C_{N,\eps}(x_N,y)$ are both equal to $N+O(1)$ by Lemma~\ref{L:covariance} and we bound
\[
\P(G_{\eps,\delta}(x,y)) \le \P(X_N(x_N) \le (\hat\gamma - 2\gamma)N + O(1)) \le C \eps^{(2\gamma-\hat\gamma)_+^2/2}.
\]
Note that the bound is trivial when $\hat\gamma>2\gamma$.
If $\eps< |x-y| < e^{-1} \delta$, 
we consider some $n \in \{n_0,\dots,N\}$ such that $e^{-1-n} \le |x-y| < e^{-n}$ and bound
\begin{align*}
\P(G_{\eps,\delta}(x,y)) \le \P(X_n(x_n) \le (\hat\gamma-2\gamma)n + O(1)) \le C e^{-(2\gamma-\hat\gamma)_+^2 n/2} \le C |x-y|^{(2\gamma-\hat\gamma)_+^2/2}
\end{align*}
which shows \eqref{E:claim_proba}.

\smallskip

\noindent\textbf{Case 1: $|x-y| \le e^3 \eps.$} In this case, the event $G_{\eps,\delta}$ imposes that $X_N(x_N) \le (\hat\gamma-2\gamma)N + C$, so we can bound $A(k,x,y,\eps,\eps)$ by
\begin{align*}
    \sum_{p \ge 0} \sigma_\eps^{2k} \E\Big[ \mathbf{1}_{\frac{X_N(x_N) - (\hat\gamma-2\gamma)N-C}{\sqrt{N}} \in (-p-1,-p] } \Big| H_k\Big( \frac{X_\eps(x)+\gamma C_\eps(x,y)}{\sigma_\eps} \Big) H_k\Big( \frac{X_\eps(y)+\gamma C_\eps(y,x)}{\sigma_\eps} \Big) \Big| \Big].
\end{align*}
By Lemma \ref{L:covariance} and because $|x-y| \le e^3 \eps$, $X_\eps(x)-X_N(x_N)$ and $X_\eps(y)-X_N(x_N)$ are Gaussians with uniformly bounded variance. On the event that $(X_N(x_N) - (\hat\gamma-2\gamma)N-C)/\sqrt{N} \in (-p-1,-p]$, $(X_\eps(x) + \gamma C_\eps(x,y))/\sigma_\eps$ is thus very concentrated around $(\hat\gamma-\gamma) \sqrt{N} -p$.
By Lemma \ref{L:Hermite_bound}, for any $\rho \in (0,1)$, there exists $C=C(\rho)$ such that for all $x \in \R$,
\[
|H_k(x)| \le C \rho^{-k/2} \sqrt{k!} e^{\frac{\rho}{2(1+\rho)}x^2}.
\]
Altogether, this leads to
\begin{align*}
    A(k,x,y,\eps,\eps) & \le C \rho^{-k} k! \sigma_\eps^{2k} \sum_{p \ge 0} e^{\frac{\rho}{1+\rho}((\hat\gamma-\gamma) \sqrt{N} -p)^2} e^{-((2\gamma-\hat\gamma)\sqrt N + p)^2/2}\\
    & \le C \rho^{-k} k! \log(C/\eps)^k \eps^{-\frac{\rho}{1+\rho}(\hat\gamma-\gamma)^2 + \frac12(2\gamma-\hat\gamma)_+^2}.
\end{align*}
We now conclude with elementary inequalities to obtain the desired form for the upper bound.
For any $\eta >0$, there exists $C=C(\eta)>0$ such that $\log(C/\eps)^k \le C(\eta) (1+\eta)^k \log(1/\eps)^k$. Furthermore, using the inequality $t^k \le k! e^t$ for any $t \ge 0$, we can bound
\[
\log(1/\eps)^k = \Big( \frac{1+\rho}{(\hat\gamma-\gamma)^2} \Big)^k \Big( \frac{(\hat\gamma-\gamma)^2}{1+\rho} \log \frac{1}{\eps} \Big)^k \le k! \Big( \frac{1+\rho}{(\hat\gamma-\gamma)^2} \Big)^k \eps^{-\frac{(\hat\gamma-\gamma)^2}{1+\rho}}.
\]
This leads to the bound
\[
A(k,x,y,\eps,\eps) \le C k!^2 \Big( \frac{(1+\eta)(1+\rho)}{\rho(\hat\gamma-\gamma)^2} \Big)^k \eps^{-(\hat\gamma-\gamma)^2 + \frac12(2\gamma-\hat\gamma)_+^2}.
\]
The term $\frac{(1+\eta)(1+\rho)}{\rho(\hat\gamma-\gamma)^2}$ can be made arbitrarily close to $2(\hat\gamma-\gamma)^{-2}$ by taking $\eta$ close enough to 0 and $\rho$ close enough to 1. This concludes the proof of \eqref{E:boundA} when $|x-y| \le e^3 \eps.$

\smallskip\noindent\textbf{Case 2: $|x-y| \ge e^3 \eps.$}
This case will be handled by using a more elaborate version of the martingale trick \eqref{E:MG_trick1}.
Let $\mathfrak{n} \in \{n_0,\dots, N-2\}$ be such that $e^{-\nf-1} \le (e^{-1}\delta \wedge |x-y| ) \le e^{-\nf}$.
Recalling the definition \eqref{E:goodgood} of the good events $G_{n,m}(x,y)$, we can write
\begin{align}\label{E:inclusion_exclusion}
\mathbf{1}_{G_{N,N}(x,y)} & = \mathbf{1}_{G_{\nf+1,\nf+1}(x,y)} + \sum_{n=\nf+1}^{N-1}  \mathbf{1}_{G_{n,\nf+1}(x,y) \cap E_{n+1}^{(1)}(x,y)^c} + \sum_{m=\nf+1}^{N-1}  \mathbf{1}_{G_{\nf+1,m}(x,y) \cap E_{m+1}^{(2)}(x,y)^c} \\
& + \sum_{n,m=\nf+1}^{N-1} \mathbf{1}_{G_{n,m}(x,y) \cap E_{n+1}^{(1)}(x,y)^c \cap E_{m+1}^{(2)}(x,y)}.\notag
\end{align}
Recall the definition \eqref{E:filtration2} of the $\sigma$-algebra $\Fc_{n,m}(x,y,\eps)$. The events $G_{n,m}(x,y)$, $E_n^{(1)}(x,y)$ and $E_m^{(2)}(x,y)$ are $\Fc_{n,m}(x,y,\eps)$-measurable. Moreover, by Lemma \ref{L:independence}, conditionally on $\Fc_{n+1,m+1}(x,y,\eps)$, $X_\eps(x)$ and $X_\eps(y)$ are independent as soon as $|x-y| \ge 2\eps + e^{-n-1} + e^{-m-1}$. One can check that this is satisfied as soon as $n \ge \nf$ and $m \ge \nf$. 
Hence, for all $n,m=\nf, \dots, N-1$,
\begin{align*}
    & \E\Big[ H_k\Big( \frac{X_\eps(x)+\gamma C_\eps(x,y)}{\sigma_\eps} \Big) H_k\Big( \frac{X_\eps(y)+\gamma C_\eps(y,x)}{\sigma_\eps} \Big) \Big| \Fc_{n+1,m+1}(x,y,\eps) \Big] \\
    & = \E\Big[ H_k\Big( \frac{X_\eps(x)+\gamma C_\eps(x,y)}{\sigma_\eps} \Big) \Big| \Fc_{n+1,m+1}(x,y,\eps) \Big] \E\Big[ H_k\Big( \frac{X_\eps(y)+\gamma C_\eps(x,y)}{\sigma_\eps} \Big) \Big| \Fc_{n+1,m+1}(x,y,\eps) \Big].
\end{align*}
Let
\[
\sigma_{\eps,n}^2 = \var( \E[X_\eps(x)\vert\Fc_{n}(x,\eps)]), \qquad n \ge 0.
\]
We can now apply Lemma \ref{L:Hermite_martingale} to
\begin{align*}
m & = \frac{ \E[X_\eps(x)\vert\Fc_{n+1}(x,\eps)]+ \gamma C_\eps(x,y) }{\sigma_\eps}
\hspace{40pt} \text{and} \\
\sigma^2 & = \frac{\var(X_\eps(x) - \E[X_\eps(x)\vert\Fc_{n+1}(x,\eps)])}{\sigma_\eps^2} = 1-\frac{\sigma_{\eps,n+1}^2}{\sigma_\eps^2}
\end{align*}
to get that for all $n,m=\nf, \dots, N-1$,
\[
\sigma_\eps^k \E\Big[ H_k\Big( \frac{X_\eps(x)+\gamma C_\eps(x,y)}{\sigma_\eps} \Big) \Big| \Fc_{n+1,m+1}(x,y,\eps) \Big] = \sigma_{\eps,n+1}^k H_k \Big( \frac{ \E[X_\eps(x)\vert\Fc_{n+1}(x,\eps)]+ \gamma C_\eps(x,y) }{\sigma_{\eps,n+1}} \Big)
\]
We have obtained that
\begin{align}\label{E:bigbig}
    & A(k,x,y,\eps,\eps) = \sigma_{\eps,\nf+1}^{2k} \E\Big[ \prod_{z=x,y} H_k \Big( \frac{ \E[X_\eps(z)\vert\Fc_{\nf+1}(z,\eps)]+ \gamma C_\eps(x,y) }{\sigma_{\eps,\nf+1}} \Big) \mathbf{1}_{G_{\nf+1,\nf+1}(x,y)} \Big] \\
    \notag
    & + \sum_{n_x=\nf+1}^{N-1} \E\Big[ \prod_{z=x,y} \Big( \sigma_{\eps,n_z+1}^k H_k \Big( \frac{ \E[X_\eps(z)\vert\Fc_{n_z+1}]+ \gamma C_\eps(x,y) }{\sigma_{\eps,n_z+1}} \Big) \Big) \mathbf{1}_{G_{n_x,\nf+1}(x,y) \cap E^{(1)}_{n_x+1}(x,y)^c} \Big] \\
    \notag
    & + \sum_{n_y=\nf+1}^{N-1} \E\Big[ \prod_{z=x,y} \Big( \sigma_{\eps,n_z+1}^k H_k \Big( \frac{ \E[X_\eps(z)\vert\Fc_{n_z+1}]+ \gamma C_\eps(x,y) }{\sigma_{\eps,n_z+1}} \Big) \Big) \mathbf{1}_{G_{\nf+1,n_y}(x,y) \cap E^{(2)}_{n_y+1}(x,y)^c} \Big] \\
    \notag
    & + \sum_{n_x,n_y=\nf+1}^{N-1} \E\Big[ \prod_{z=x,y} \Big( \sigma_{\eps,n_z+1}^k H_k \Big( \frac{ \E[X_\eps(z)\vert\Fc_{n_z+1}]+ \gamma C_\eps(x,y) }{\sigma_{\eps,n_z+1}} \Big) \Big) \mathbf{1}_{G_{n_x,n_y}(x,y) \cap E^{(1)}_{n_x+1}(x,y)^c \cap E^{(2)}_{n_y+1}(x,y)^c} \Big]
\end{align}
where in the second (resp. third) term, $n_y$ (resp. $n_x$) is by convention set to $\nf$.

Let's focus on the last right hand side term of \eqref{E:bigbig}.
Let $n_x,n_y \in \{\nf,\dots, N-1\}$. As before, on the event $G_{n_x,n_y}(x,y)$,
\[ 
\frac{ \E[X_\eps(z)\vert\Fc_{n_z+1}(z,\eps)]+ \gamma C_\eps(x,y) }{\sigma_{\eps,n_z+1}}
\text{ concentrates around } (\hat\gamma-\gamma) \sqrt{n_z}.
\]
Using Lemma \ref{L:Hermite_bound}, we get that
\begin{align*}
    & \E\Big[ \prod_{z=x,y} \sigma_{\eps,n_z+1}^k H_k \Big( \frac{ \E[X_\eps(z)\vert\Fc_{n_z+1}]+ \gamma C_\eps(x,y) }{\sigma_{\eps,n_z+1}} \Big) \mathbf{1}_{G_{n_x,n_y}(x,y) \cap E^{(1)}_{n_x+1}(x,y)^c \cap E^{(2)}_{n_y+1}(x,y)^c} \Big] \\
    & \le C n_x^{k/2} n_y^{k/2} \rho^{-k} k! e^{\frac{\rho}{2(1+\rho)}(\hat\gamma-\gamma)^2 (n_x+n_y)} \P (G_{n_x,n_y}(x,y) \cap E^{(1)}_{n_x+1}(x,y)^c \cap E^{(2)}_{n_y+1}(x,y)^c).
\end{align*}
On the event $G_{n_x,n_y}(x,y) \cap E^{(1)}_{n_x+1}(x,y)^c \cap E^{(2)}_{n_y+1}(x,y)^c$, $X_\nf(x) \le (\hat\gamma-2\gamma)\nf+O(1)$ and for all $z\in \{x,y\}$, $X_{n_z+1}(z)-X_\nf(z) \ge (\hat\gamma-\gamma)(n_z-\nf)+O(1)$. Since these random variables are independent, we deduce that
\begin{align*}
\P(G_{n_x,n_y}(x,y) \cap E^{(1)}_{n_x+1}(x,y)^c \cap E^{(2)}_{n_y+1}(x,y)^c)
& \le C e^{-\frac12(\hat\gamma-\gamma)^2 (n_x+n_y-2\nf)} e^{-\frac12(2\gamma-\hat\gamma)_+^2 \nf}.
\end{align*}
Notice that, when $\hat\gamma\ge2\gamma$, we simply bound the probability that $X_\nf(x) \le (\hat\gamma-2\gamma)\nf+O(1)$ by one.
Putting things together, the last term on the right hand side of \eqref{E:bigbig} is at most
\begin{align*}
    C \rho^{-k} k! e^{-( \frac12(2\gamma-\hat\gamma)_+^2 - \frac{\rho}{1+\rho} (\hat\gamma-\gamma)^2 ) \nf} \Big( \sum_{n=\nf+1}^{N-1} n^{k/2} e^{-\frac12(1-\frac{\rho}{1+\rho})(\hat\gamma-\gamma)^2(n-\nf)} \Big)^2.
\end{align*}
Comparing the sum with an integral and then integrating by parts successively, we have
\begin{align*}
    \sum_{n=\nf+1}^{N-1} n^{k/2} e^{-\frac12(1-\frac{\rho}{1+\rho})(\hat\gamma-\gamma)^2(n-\nf)} 
    \le C \int_{\nf+1}^\infty u^{k/2} e^{-\frac{(\hat\gamma-\gamma)^2}{2(1+\rho)}(u-\nf)} \d u
    \le C \sum_{i=0}^{k/2} \frac{(k/2)!}{(k/2-i)!} \nf^{k/2-i} \Big( \frac{2(1+\rho)}{(\hat\gamma-\gamma)^2} \Big)^i.
\end{align*}
So the last right hand side term of \eqref{E:bigbig} is at most
\begin{align*}
    & C \rho^{-k} k! e^{-( \frac12(2\gamma-\hat\gamma)_+^2 - \frac{\rho}{1+\rho} (\hat\gamma-\gamma)^2 ) \nf} \Big( \sum_{i=0}^{k/2} \frac{(k/2)!}{(k/2-i)!} \nf^{k/2-i} \Big( \frac{2(1+\rho)}{(\hat\gamma-\gamma)^2} \Big)^i \Big)^2 \\
    & \le C \rho^{-k} k! (k/2)!^2 e^{-( \frac12(2\gamma-\hat\gamma)_+^2 - \frac{\rho}{1+\rho} (\hat\gamma-\gamma)^2 ) \nf} \Big( \frac{2(1+\rho)}{(\hat\gamma-\gamma)^2} \Big)^{k} \exp \left( \frac{(\hat\gamma-\gamma)^2}{1+\rho} \nf \right) \\
    & = C k! (k/2)!^2 e^{-( \frac12(2\gamma-\hat\gamma)_+^2 - (\hat\gamma-\gamma)^2 ) \nf} \Big( \frac{2(1+\rho)}{\rho(\hat\gamma-\gamma)^2} \Big)^{k}.
\end{align*}
By Stirling's formula, $(k/2)!^2 \sim \sqrt{\frac{\pi}{2}k} 2^{-k} k!$, so the above term is bounded by
\[
C \sqrt{k} k!^2 e^{-( \frac12(2\gamma-\hat\gamma)_+^2 - (\hat\gamma-\gamma)^2 ) \nf} \Big( \frac{1+\rho}{\rho(\hat\gamma-\gamma)^2} \Big)^{k}.
\]
For any $u > \sqrt2/(\hat\gamma-\gamma)$, by taking $\rho$ close enough to 1, we can ensure the existence of a constant $C=C(u)>0$ such that for all $k \ge 1$,
\[
\sqrt{k} \Big( \frac{1+\rho}{\rho(\hat\gamma-\gamma)^2} \Big)^{k} \le u^{2k}.
\]
Overall, this shows that the last right hand side term of \eqref{E:bigbig} is at most
\[ 
C k!^2 u^{2k} e^{-( \frac12(2\gamma-\hat\gamma)_+^2 - (\hat\gamma-\gamma)^2 ) \nf}.
\]
The other three right hand side terms of \eqref{E:bigbig} can be handled in a similar manner resulting in upper bounds controlled by the above display. This concludes the proof of \eqref{E:boundA}.

\medskip

The proof of the convergence of $A(k,x,y,\eps,\eps)$ for fixed distinct $x,y\in[0,1]^d$ also follows. Indeed, we can decompose $A(k,x,y,\eps,\eps)$ exactly as in \eqref{E:bigbig} but where $\nf$ is this time a large integer (larger than $-\log|x-y|$). From similar computations as we did above, the sum of the last three right hand side term of \eqref{E:bigbig} can be bounded by $c(\nf)$, uniformly in $\eps$, where $c(\nf)\to 0$ as $\nf\to \infty$. For $\nf$ fixed, the first right hand side term of \eqref{E:bigbig} converges as $\eps\to0$ to
\[
\sigma_{\nf+1}^{2k} \E\Big[ \prod_{z=x,y} H_k \Big( \frac{ \E[X(z)\vert\Fc_{\nf+1}(z,0)]+ \gamma C(x,y) }{\sigma_{\nf+1}} \Big) \mathbf{1}_{G_{\nf+1,\nf+1}(x,y)} \Big].
\]
This proves that
\[
\limsup_{\eps_1,\eps_2\to 0} |A(k,x,y,\eps_1,\eps_1) - A(k,x,y,\eps_2,\eps_2)| \le c(\nf).
\]
Since the left hand side does not depend on $\nf$, it must vanish. The sequence $(A(k,x,y,\eps,\eps))_{\eps\in (0,\delta)}$ is thus Cauchy and hence converges. This concludes the proof.
\end{proof}

We are now ready to prove Proposition \ref{P:bounded_L2}.

\begin{proof}[Proof of Proposition \ref{P:bounded_L2}]
Recall from \eqref{E:linkAI} that
\begin{align*}
\E[I_{\eps,\delta}^{\af,\good}(f,k)^2]
& = \int_{[0,1]^d \times [0,1]^d} f(x)f(y) e^{\gamma^2 C_\eps(x,y)} A(k,x,y,\eps,\eps) \d x \d y.
\end{align*}
By Lemma \ref{L:boundA},
the contribution of $|x-y| < \delta$ is bounded by
\begin{align*}
    & C u^{2k} k!^2 \|f\|_\infty^2 \int_{[0,1]^d \times [0,1]^d} \mathbf{1}_{|x-y| \le \delta} (|x-y| \vee \eps)^{-\gamma^2 + \frac12 (2\gamma-\hat\gamma)_+^2-(\hat\gamma-\gamma)^2} \d x \d y.
\end{align*}
When $\hat\gamma < \sqrt{2d}\wedge(2\gamma)$, the exponent $-\gamma^2 + \frac12 (2\gamma-\hat\gamma)_+^2-(\hat\gamma-\gamma)^2$ is larger than $-d$.
In that case, the integral is uniformly bounded in $\eps$ and can be bounded by a constant depending only on $\gamma$ and $\hat\gamma$.
Using Lemma \ref{L:boundA}, 
the contribution of $|x-y| > \delta$ to the integral is bounded by
$C u^{2k} k!^2 \|f\|_\infty^2 \delta^{-\gamma^2 + \frac12 (2\gamma-\hat\gamma)^2-(\hat\gamma-\gamma)^2}$.
This concludes the proof of \eqref{E:P_bdd_L2}.

We now wish to show that $(I_{\eps,\delta}^{\af,\good}(f,k))_{\eps\in (0,\delta)}$ is Cauchy in $L^2$. We expand
\begin{align*}
    \E[(I_{\eps_1,\delta}^{\af,\good}(f,k)-I_{\eps_2,\delta}^{\af,\good}(f,k))^2]
    & = \E[I_{\eps_1,\delta}^{\af,\good}(f,k)^2] + \E[I_{\eps_2,\delta}^{\af,\good}(f,k)^2] \\
    & - 2 \E[I_{\eps_1,\delta}^{\af,\good}(f,k)I_{\eps_2,\delta}^{\af,\good}(f,k)]
\end{align*}
and use \eqref{E:linkAI} to write the three right hand side expectations as integrals involving $A(k,x,y,\eps,\eps')$ with $\eps,\eps'\in\{\eps_1,\eps_2\}$. Lemma \ref{L:boundA} provides the domination we need to apply dominated convergence theorem and we get that
\begin{align*}
    & \limsup_{\eps_1,\eps_2\to 0} \E[(I_{\eps_1,\delta}^{\af,\good}(f,k)-I_{\eps_2,\delta}^{\af,\good}(f,k))^2]\\
    & = \hspace{-2pt} \int_{[0,1]^d \times [0,1]^d} \hspace{-23.5pt} f(x)f(y) \limsup_{\eps_1,\eps_2\to 0} \Big( \sum_{j=1}^2 e^{\gamma^2 C_{\eps_j}(x,y)} A(k,x,y,\eps_j,\eps_j) - 2 e^{\gamma^2 C_{\eps_1,\eps_2}(x,y)} A(k,x,y,\eps_1,\eps_2) \Big) \d x \d y.
\end{align*}
By Lemma \ref{L:boundA} and \ref{L:covariance} respectively, both $A(k,x,y,\eps_1,\eps_2)$ and $C_{\eps_1,\eps_2}(x,y)$ converge pointwise as $\eps_1,\eps_2\to 0$. This shows that $(I_{\eps,\delta}^{\af,\good}(f,k))_{\eps\in (0,\delta)}$ is Cauchy which concludes the proof.
\end{proof}

\subsection{Higher moment, good event}\label{SS:higher}

In this section, we prove Proposition \ref{P:bound_pmoment}. Since it is very similar to the proof of Proposition \ref{P:bounded_L2}, we only sketch the proof.

\begin{proof}[Sketch of proof of Proposition \ref{P:bound_pmoment}]
    The restriction $\gamma < \sqrt{2d/p}$ corresponds to the phase where GMC possesses finite $p$-moment (see e.g. \cite[Theorem 2.11]{zbMATH06370363}) although, as we will see, we could consider slightly larger values of $\gamma$ thanks to the presence of the good event.
    To ease notations, let us consider the third moment, i.e. bound
    \begin{align}\label{E:p=2a}
        \E[I_{\eps,\delta}^{\af,\good}(f,k)^3]
        =\int_{([0,1]^d)^3} f(x)f(y)f(z) \E\Big[\prod_{w=x,y,z}
        \wick{X_\eps(w)^k e^{\gamma X_\eps(w)}}\mathbf{1}_{G_{\eps,\delta}(w)}\Big] \d x \d y \d z.
    \end{align}
    Let $x,y,z\in [0,1]^d$. Denote by $r_x = |x-y|\wedge|x-z|$, $r_y = |x-y|\wedge|y-z|$ and $r_z = |x-z|\wedge|y-z|$.
    Let us focus on the main case where for all $w=x,y,z$, $r_w \in (\eps,\delta)$. Without loss of generality, we may assume that $r_z \ge r_x\vee r_y$, so that $x$ and $y$ are ``close to each other'', while $z$ is further away.
    By Cameron--Martin formula, we can rewrite
    \begin{align}\label{E:p=2b}
        & \E\Big[\prod_{w=x,y,z}
        \wick{X_\eps(w)^k e^{\gamma X_\eps(w)}}\mathbf{1}_{G_{\eps,\delta}(w)}\Big]
        = e^{\gamma^2\E[X_\eps(x)X_\eps(y) + X_\eps(y)X_\eps(z)+X_\eps(z)X_\eps(x)]} \times \\
        & \hspace{40pt} \times
        \E\Big[\prod_{w=x,y,z} H_k \Big( \frac{X_\eps(w) +\gamma \sum_{w' \in \{x,y,z\}\setminus\{w\}} C_\eps(w,w') }{\sigma_\eps} \Big) \mathbf{1}_{G_{N,N,N}(x,y,w)} \Big],
        \notag
    \end{align}
    where
    \begin{align*}
        G_{N_1,N_2,N_3}(x,y,w) = \bigcap_{n=n_0}^{N_1} E_n^{(1)}(x,y,z) \cap \bigcap_{n=n_0}^{N_2} E_n^{(2)}(x,y,z) \cap \bigcap_{n=n_0}^{N_3} E_n^{(3)}(x,y,z)
    \end{align*}
    where for all $n \in \{n_0,\dots,N\}$,
\begin{align}\nonumber
E_n^{(1)}(x,y,z) & = 
 \{ X_n(x_n) \le \hat\gamma n - \gamma \sum_{w=x,y,z} C_{n,\eps}(x_n,w) \}, \\
E_n^{(2)}(x,y,z) & = \{ X_n(y_n) \le \hat\gamma n - \gamma \sum_{w=x,y,z} C_{n,\eps}(y_n,w) \}, \\
E_n^{(3)}(x,y,z) & = \{ X_n(z_n) \le \hat\gamma n - \gamma \sum_{w=x,y,z} C_{n,\eps}(z_n,w) \}.
\end{align}
As in the $p=2$ case, we decompose the event $G_{N,N,N}(x,y,z)$ using an inclusion-exclusion sum; see \eqref{E:inclusion_exclusion}. When there were only two points $x,y$, the targeted largest scale was denote by $\nf$ and was such that $e^{-\nf-1}\le|x-y| \le e^{-\nf}$. In the present case, we consider integers $\nf_w$ such that $e^{-\nf_w-1}\le r_w \le e^{-\nf_w}$, for $w=x,y,z$. Using our martingale trick, we need to estimate
\begin{align}\label{E:p=3}
    & \sum_{\substack{\nf_w+1\le n_w\le N-1\\w=x,y,z}} \E\Big[ \prod_{w=x,y,z} \Big( \sigma_{\eps,n_w+1}^k H_k \Big( \frac{ \E[X_\eps(w)\vert\Fc_{n_w+1}]+\gamma \sum_{w' \in \{x,y,z\}\setminus\{w\}} C_\eps(w,w') }{\sigma_{\eps,n_w+1}} \Big) \Big) \times \\
    & \hspace{40pt} \times \mathbf{1}_{G_{n_x,n_y,n_z}(x,y,z) \cap E^{(1)}_{n_x+1}(x,y,z)^c \cap E^{(2)}_{n_y+1}(x,y,z)^c \cap E^{(3)}_{n_z+1}(x,y,z)^c} \Big]
    \notag
\end{align}
as well as other boundary terms with a similar form; see \eqref{E:bigbig}. As in the $p=2$ case, we can bound the above expectation by
\begin{align*}
    & C \prod_{w=x,y,z} \big( n_w^{k/2} \rho^{-k/2} \sqrt{k!} e^{\frac{\rho}{2(1+\rho)}(\hat\gamma-\gamma)^2 n_w} \big) \\
    & \times \P(G_{n_x,n_y,n_z}(x,y,z) \cap E^{(1)}_{n_x+1}(x,y,z)^c \cap E^{(2)}_{n_y+1}(x,y,z)^c \cap E^{(3)}_{n_z+1}(x,y,z)^c).
\end{align*}
On the event $G_{n_x,n_y,n_z}(x,y,z) \cap E^{(1)}_{n_x+1}(x,y,z)^c \cap E^{(2)}_{n_y+1}(x,y,z)^c \cap E^{(3)}_{n_z+1}(x,y,z)^c$, we have $X_{\nf_z}(z) \le (\hat\gamma-3\gamma)\nf_z+O(1)$ (recall that we have assumed that $r_z \ge r_x\vee r_y$), and for all $w=x,y,z$,
$X_{n_w+1}(w)-X_{\nf_z}(z) \ge (\hat\gamma-\gamma)(n_w-\nf_z)+O(1)$ and we get
\begin{align*}
& \P(G_{n_x,n_y,n_z}(x,y,z) \cap E^{(1)}_{n_x+1}(x,y,z)^c \cap E^{(2)}_{n_y+1}(x,y,z)^c \cap E^{(3)}_{n_z+1}(x,y,z)^c)\\
& \le C e^{-\frac12(3\gamma-\hat\gamma)_+^2 \nf_z} e^{-\frac12(\hat\gamma-\gamma)^2(\nf_x-\nf_z)} \prod_{w=x,y,z} e^{-\frac12(\hat\gamma-\gamma)^2 (n_w-\nf_w)}.
\end{align*}
The sum displayed in \eqref{E:p=3} is thus at most
\begin{align*}
    & C \rho^{-3k/2} k!^{3/2} e^{-\frac12(3\gamma-\hat\gamma)_+^2 \nf_z} e^{-\frac12(\hat\gamma-\gamma)^2(\nf_x-\nf_z)} e^{\frac{\rho}{2(1+\rho)}(\hat\gamma-\gamma)^2 (\nf_x+\nf_y+\nf_z)} \\
    & \times \prod_{w=x,y,z} \Big(\sum_{n_w=\nf_w+1}^{N-1} n_w^{k/2} e^{-\frac1{2(1+\rho)}(\hat\gamma-\gamma)^2 (n_w-\nf_w)} \Big) \\
    & \le C \rho^{-3k/2} k!^{3/2} e^{-\frac12(3\gamma-\hat\gamma)_+^2 \nf_z}  e^{-\frac12(\hat\gamma-\gamma)^2(\nf_x-\nf_z)} e^{\frac{\rho}{2(1+\rho)}(\hat\gamma-\gamma)^2 (\nf_x+\nf_y+\nf_z)} \\
    & \times \prod_{w=x,y,z} \Big((k/2)!\Big( \frac{\sqrt{2(1+\rho)}}{\hat\gamma-\gamma} \Big)^k e^{\frac{(\hat\gamma-\gamma)^2}{2(1+\rho)}\nf_w} \Big)\\
    & \le C k!^3 \Big( \frac{\sqrt{(1+\rho)}}{\rho(\hat\gamma-\gamma)} \Big)^{3k} e^{-\frac12(3\gamma-\hat\gamma)_+^2 \nf_z} e^{-\frac12(\hat\gamma-\gamma)^2(\nf_x-\nf_z)} e^{\frac{(\hat\gamma-\gamma)^2}{2} (\nf_x+\nf_y+\nf_z)}.
\end{align*}
Going back to \eqref{E:p=2b}, we will have to multiply the above by
\[
e^{\gamma^2\E[X_\eps(x)X_\eps(y) + X_\eps(y)X_\eps(z)+X_\eps(z)X_\eps(x)]}
\le C e^{2\gamma^2 \nf_z} |x-y|^{-\gamma^2}
\]
and integrate the resulting quantity over $x,y,z\in[0,1]^d$ with $|x-y|\le |x-z|\wedge|y-z|$. This leads to the following bound
\begin{align*}
    & \E[I_{\eps,\delta}^{\af,\good}(f,k)^3]
    \le C \|f\|_\infty^3 k!^3 \Big( \frac{\sqrt{(1+\rho)}}{\rho(\hat\gamma-\gamma)} \Big)^{3k} \sum_{\nf_z \ge 0} \int_{[0,1]^d} \d z \int_{B(z,e^{-\nf_z})\setminus B(z,e^{-\nf_z-1})} \hspace{-6pt} \d x \int_{B(x,e^{-\nf_z})} \d y \times \\
    & \hspace{20pt} \times e^{-\frac12(3\gamma-\hat\gamma)_+^2 \nf_z + (\hat\gamma-\gamma)^2\nf_z+2\gamma^2\nf_z}
    |x-y|^{-\gamma^2 -\frac12(\hat\gamma-\gamma)^2}.
\end{align*}
When $\hat\gamma<\sqrt{2d}$, the exponent $\gamma^2+\frac12(\hat\gamma-\gamma)^2$ is smaller than $d$ for all $\gamma < \frac{2}{3}\sqrt{2d}$. Since we have assumed that $\gamma < \sqrt{2d/3}$, this is indeed the case here and the integral over $y$ of $|x-y|^{-\gamma^2-\frac12(\hat\gamma-\gamma)^2}$ is bounded by $C e^{-(d-\gamma^2-\frac12(\hat\gamma-\gamma)^2)\nf_z}$.
The integral over $x$ will contribute $e^{-d\nf_z}$ and the sum over $\nf_z$ can be bounded by
\[
C \sum_{\nf_z\ge0} e^{-(2d-\gamma^2-\frac12(\hat\gamma-\gamma)^2 +\frac12(3\gamma-\hat\gamma)_+^2 - (\hat\gamma-\gamma)^2-2\gamma^2)\nf_z}.
\]
One can check that the function $\hat\gamma\in(\gamma,3\gamma\wedge\sqrt{2d})\mapsto 2d-\gamma^2-\frac12(\hat\gamma-\gamma)^2 +\frac12(3\gamma-\hat\gamma)_+^2 - (\hat\gamma-\gamma)^2-2\gamma^2$ stays positive (it vanishes at $\hat\gamma=\sqrt{2d}$, when $3\gamma\ge\sqrt{2d}$). The sum over $\nf_z$ is thus bounded by some constant which concludes the proof.
\end{proof}

\subsection{First moment, bad event}\label{SS:1st}

Let $\eta>0$ and $\hat\gamma\in (\gamma,\sqrt{2d})$.
For $N \ge 0$, let
\begin{equation}\label{E:defMN}
    M_N = \max_{\substack{\eps = e^{-n}\\0 \le n \le N}} \max_{k\ge 0} \frac{1}{k!} \Big(  \frac{(1+\eta)\sqrt{2}}{\hat\gamma-\gamma} \Big)^{-k} \hspace{-3pt}\sup_{\af \in [0,1]} \sup_{0<\|f\|_\infty < \infty} \frac1{\|f\|_\infty} \E \Big[ \Big| \int_{[0,1]^d} \hspace{-3.8pt}  f(x) :X_\eps^\af(x)^k e^{\gamma X_\eps^\af(x)}: \d x \Big| \Big]
\end{equation}
where the supremum runs over scaling parameters $\af \in [0,1]$ and measurable functions $f:[0,1]^d \to \R$ with $0<\|f\|_\infty < \infty$. We will show that the sequence $(M_N)_{N \ge0}$ is bounded.
Let us start by showing that it is a sequence of finite real numbers.

\begin{lemma}\label{L:initialisation}
    For all $N \ge 0$, $M_N$ is finite.
\end{lemma}

This result is extremely crude since it does not control the potential blow-up of $M_N$ as the approximation level $N$ goes to infinity. Nevertheless, it will serve as an initial estimate in our inductive argument below.

\begin{proof}[Proof of Lemma \ref{L:initialisation}]
    Let $n \ge 0$, $\eps = e^{-n}$. Let $f:[0,1]^d \to \R$ be bounded and $\af \in [0,1]$. Since we do not wish a uniform control in $n$, we can use the trivial bound
    \begin{align}\label{E:pf_initi1}
        \E \Big[ \Big| \int_{[0,1]^d}  f(x) :X_\eps^\af(x)^k e^{\gamma X_\eps^\af(x)}: \d x \Big| \Big]
        \le \|f\|_\infty \int_{[0,1]^d} \E[ |:X_\eps^\af(x)^k e^{\gamma X_\eps^\af(x)}:|] \d x.
    \end{align}
    By Cameron--Martin theorem, the expectation on the right hand side equals
    \begin{align*}
        (\sigma_\eps^\af)^k \E \Big[ \Big| H_k \Big( \frac{X_\eps^\af(x)}{\sigma_\eps^\af} \Big) \Big| \Big] = (\sigma_\eps^\af)^k \E[|H_k(Z)|]
    \end{align*}
    where $Z$ is a standard normal random variable. Lemma \ref{L:Hermite_bound} applied to $\rho =1/2$ shows that $\E[|H_k(Z)|] / \sqrt{k!}$ grows at most exponentially in $k$. By Lemma \ref{L:covariance}, the term $(\sigma_\eps^\af)^k$ is bounded by $C(n)^k$, uniformly in $\af$. The left hand side of \eqref{E:pf_initi1} is thus at most $C(n)^k \sqrt{k!} \|f\|_\infty$ which implies that $M_N$ is finite for all $N\ge0$.
\end{proof}

The following result is at the heart of our inductive argument.

\begin{lemma}\label{L:inductive}
Let $\rho \in (0,1)$.
There exists $C,c>0$ depending possibly on $\gamma$, $\hat\gamma$ and $\rho$ such that for all $\af \in [0,1]$, $N \ge \mathfrak{m} \ge 1$, $k \ge 0$ and $f:\R^d \to \R$ a measurable bounded function supported in $[0,1]^d$, the following holds.
Let $\delta = e^{-\mathfrak{m}}$, $\eps=e^{-N}$ and for $n \ge 0$, let $\af_n = \af e^{-n\alpha}$. 
For all $n=\mathfrak{m},\dots, N$ and $j=0,\dots, k$, there exist a random function $\tilde f_{j,n} : \R^d \to \R$ supported in $[0,1]^d$ and a centred Gaussian random variable $N_n$ with variance bounded by $n+C$, jointly independent of the field $X_{e^n \eps}^{\af_n}$ such that
\begin{equation}\label{E:L_inductive1}
    \E[ |I_{\eps,\delta}^{\af,\bad}(f,k)|] \le C \sum_{n=\mathfrak{m}}^N \sum_{j=0}^k \binom{k}{j} \E[\mathbf{1}_{N_n \ge (\hat\gamma-\gamma)n+C} |I_{e^n \eps}^{\af_n}(\tilde f_{k-j,n},j)| ].
\end{equation}
Moreover, for all $n=\mathfrak{m},\dots, N$ and $j=0,\dots, k$, almost surely,
\begin{equation}\label{E:L_inductive2}
\E[\|\tilde f_{k-j,n}\|_\infty \vert N_n] \le C \|f\|_\infty n^{(k-j)/2} \rho^{-(k-j)} \sqrt{(k-j)!} e^{\frac{1+C/n}{4n} N_n^2}.
\end{equation}
\end{lemma}

In the above statement, we are not very specific about what we mean by ``random function''. As we will see in the proof, they are simply measurable functions of smooth Gaussian fields.

\begin{proof}[Proof of Lemma \ref{L:inductive}]
Let $\rho \in (0,1)$.
Let $N \ge \mathfrak{m} \ge 1$, $\af \in [0,1]$ and $f : [0,1]^d \to \R$ be a measurable function and $k \ge 0$. Set $\delta = e^{-\mathfrak{m}}$ and $\eps=e^{-N}$.
We decompose the ``bad event'' according to the first scale where the condition that $X^\af_n(x_n) \le \hat\gamma n$ fails. More precisely, for $x \in \R^d$, we can write
\begin{align}
G_{\eps,\delta}(x)^c = \bigsqcup_{n= \mathfrak{m}}^N F_{n,\delta}(x) \quad \text{where} \quad
F_{n,\delta}(x) = E_n(x_n)^c \cap \bigcap_{m=\mathfrak{m}}^{n-1} E_m(x_m), \quad n \in \{\mathfrak{m},\dots, N\}.
\end{align}
We can write $I_{\eps,\delta}^\bad$ using this partition and then use that $F_{n,\delta}(x)=F_{n,\delta}(y)$ for all $y \in \Lambda_n$ and $x \in Q_n(y)$ (recall Notation \ref{N:lattice}) to get that
\begin{align*}
I_{\eps,\delta}^{\af,\bad}(f,k) & = \sum_{n=\mathfrak{m}}^N \int_Q f(x) : X^\af_\eps(x)^k e^{\gamma X^\af_\eps(x)}: \mathbf{1}_{F_{n,\delta}(x)} \d x \\
& = \sum_{n= \mathfrak{m}}^N \sum_{y \in \Lambda_n} \mathbf{1}_{F_{n,\delta}(y)} \int_{Q_n(y)} f(x) : X^\af_\eps(x)^k e^{\gamma X^\af_\eps(x)}:  \d x.
\end{align*}
We can thus bound
\begin{gather}\label{E:expect1}
\E[|I_{\eps,\delta}^{\af,\bad}(f,k)|] \le \sum_{n= \mathfrak{m}}^N e^{-dn} \sum_{y \in \Lambda_n} E_{n,y},\\
\text{where} \hspace{69pt} E_{n,y} = e^{dn} \E \Big[ \mathbf{1}_{F_{n,\delta}(y)} \Big| \int_{Q_n(y)} f(x) : X^\af_\eps(x)^k e^{\gamma X^\af_\eps(x)}:  \d x \Big| \Big]. \hspace{75pt}
\notag
\end{gather}
We bound $\mathbf{1}_{F_{n,\delta}(y)} \le \mathbf{1}_{\{X^\af_n(y) \ge \hat\gamma n\}}$, do the change of variable $x = y + e^{-n} x'$ and get that
\begin{align*}
    E_{n,y} \le \E \Big[ \mathbf{1}_{\{X^\af_n(y) \ge \hat\gamma n \}} \Big| \int_{Q} f(y+e^{-n}x') : X^\af_\eps(y+e^{-n}x')^k e^{\gamma X^\af_\eps(y+e^{-n}x')}:  \d x' \Big| \Big].
\end{align*}
We now use Lemma \ref{L:scaling_covariance_Hermite}. Letting $\af_n = \af e^{-n\alpha}$, $(: X_\eps^\af(y+e^{-n} \cdot)^k e^{\gamma X_\eps^\af(y+e^{-n} \cdot)}:, X_n^\af(y))$ has the same law as
\[
    \Big( \sum_{j=0}^k \binom{k}{j} :X_{e^n \eps}^{\af_n}(\cdot)^j e^{\gamma X_{e^n \eps}^{\af_n}(\cdot)}: \times
    :X_{n,\eps}^\af(e^{-n}\cdot)^{k-j} e^{\gamma X_{n,\eps}^\af(e^{-n}\cdot)}:, X_n^\af(0) \Big)
\]
where the fields $X_{e^n\eps}^{\af_n}$ and $X_n^\af$ are independent and $X_{n,\eps}^\af = X_n^\af * \varphi_\eps$.
Let $N_n = X^\af_n(0)$ and for $x \in \R^d$, let
\begin{align*}
    f_j(x) = f(y + e^{-n}x) (\sigma_{n,\eps}^\af)^j H_j \Big( \frac{X^\af_{n,\eps}(e^{-n}x) - \gamma (\sigma_{n,\eps}^\af)^2}{\sigma_{n,\eps}^\af} \Big) e^{\gamma(X_{n,\eps}^\af(e^{-n}x)-N_n) -\frac{\gamma^2}{2}((\sigma_{n,\eps}^\af)^2 - (\sigma_n^\af)^2)} \mathbf{1}_{x \in Q}.
\end{align*}
Then
\begin{align*}
    E_{n,y} \le & \sum_{j=0}^k \binom{k}{j} \E\Big[ \mathbf{1}_{\{ N_n \ge \hat\gamma n \}} e^{\gamma N_n - \frac{\gamma^2}{2}(\sigma_n^\af)^2} |I_{e^n \eps}^{\af_n}(f_{k-j},j)| \Big].
\end{align*}
By Cameron--Martin theorem, we further have
\begin{align*}
    E_{n,y} \le \sum_{j=0}^k \binom{k}{j} \E\Big[ \mathbf{1}_{\{ N_n \ge \hat\gamma n -\gamma (\sigma_n^\af)^2 \}} |I_{e^n \eps}^{\af_n}(\tilde f_{k-j},j)| \Big]
\end{align*}
where we define for $x \in \R^d$, $m_{n,\eps}(x) = \E[X_n^\af(0)X_{n,\eps}^\af(e^{-n}x)] - \frac12 (\sigma_n^\af)^2 - \frac12(\sigma_{n,\eps}^\af)^2$, $\tilde m_{n,\eps}(x) = \E[X_n^\af(0)X_{n,\eps}^\af(e^{-n}x)] - (\sigma_{n,\eps}^\af)^2$ and
\[
\tilde f_{j}(x) = f(y + e^{-n} x) (\sigma_{n,\eps}^\af)^j H_j \Big( \frac{X^\af_{n,\eps}(e^{-n}x) + \gamma \tilde m_{n,\eps}(x)}{\sigma_{n,\eps}^\af} \Big) e^{\gamma(X_{n,\eps}^\af(e^{-n}x)-N_n) +\gamma^2 m_{n,\eps}(x)} \mathbf{1}_{x \in Q}.
\]
Going back to \eqref{E:expect1}, this concludes the proof of \eqref{E:L_inductive1}.

To finish the proof of the lemma, we need to control the $L^\infty(Q)$-norm of $\tilde f_j$. 
First, by \eqref{E:L_covariance5} and \eqref{E:L_covariance6}, as $x$ varies in $Q$,
\[
m_{n,\eps}(x), \quad \tilde m_{n,\eps}(x) \quad \text{and} \quad \E[X_{n,\eps}^\af(e^{-n} x) X_n^\af(0)] - (\sigma_n^\af)^2
\]
are uniformly bounded, independently of $\af$, $n$ and $\eps$.
Then, by using Lemma \ref{L:Hermite_bound} to bound the Hermite polynomial, we get that $\|\tilde f_{j}\|_{L^\infty(Q)}$ is at most
\begin{align*}
    C \|f\|_\infty n^{j/2} \rho^{-j/2} \sqrt{j!} \exp \Big( \frac{\rho}{2(1+\rho)}\frac{1}{n} (\|X^\af_{n,\eps}(e^{-n}\cdot)\|_{L^\infty(Q)} + C)^2 + \gamma \|X^\af_{n,\eps}(e^{-n}\cdot)-N_n\|_{L^\infty(Q)} \Big).
\end{align*}
We can decompose
\[
X_{n,\eps}^\af(e^{-n}x) - N_n = \frac{\E[X_{n,\eps}^\af(e^{-n} x) X_n^\af(0)] - (\sigma_n^\af)^2}{(\sigma_n^\af)^2} N_n + Z(x)
\]
where $Z$ is a Gaussian field, independent of $N_n$.
So
\begin{align*}
    \|X^\af_{n,\eps}(e^{-n}\cdot)-N_n\|_{L^\infty(Q)} \hspace{-1pt}\le\hspace{-1pt} \tfrac{C}{n} |N_n| + \|Z\|_{L^\infty(Q)};
    \hspace{4pt} \|X^\af_{n,\eps}(e^{-n}\cdot)\|_{L^\infty(Q)} \hspace{-1pt}\le\hspace{-1pt} \big(1+\tfrac{C}{n}\big) |N_n| + \|Z\|_{L^\infty(Q)}.
\end{align*}
Let $\eta =\frac{1-\rho}{2\rho}$.
We can bound $(a+b)^2 \le (1+\eta)a^2 + (1+\frac1\eta)b^2$ for all $a,b\ge 0$. Because $(1+\eta) \frac{\rho}{2(1+\rho)} = \frac14$,
this allows us to bound
\[
\frac{\rho}{2(1+\rho)} \frac{1}{n} (\|X^\af_{n,\eps}(e^{-n}\cdot)\|_{L^\infty(Q)} + C)^2
\le \frac{1+C/n}{4n} N_n^2 + \frac{C}{n} \|Z\|_{L^\infty(Q)}^2 + C
\]
Overall,
\begin{align*}
    \|\tilde f_{j}\|_{L^\infty(Q)} \le C \|f\|_\infty n^{j/2} \rho^{-j/2} \sqrt{j!} e^{\frac{1+C/n}{4n} N_n^2} e^{\frac{C}{n} \|Z\|_{L^\infty(Q)}^2 + \gamma \|Z\|_{L^\infty(Q)}}.
\end{align*}
Since
$\sup_{x \in Q} \E[Z(x)^2]$ is uniformly bounded in $\af$, $n$ and $\eps$, by the Borell--TIS inequality, $\|Z\|_{L^\infty(Q)}$ has Gaussian tails (with uniform constants). So the expectation of $e^{\frac{C}{n} \|Z\|_{L^\infty(Q)}^2 + \gamma \|Z\|_{L^\infty(Q)}}$ is uniformly bounded. This concludes the proof of \eqref{E:L_inductive2}.
\end{proof}

We now have all the ingredients to prove Proposition \ref{P:bad_L1}.

\begin{proof}[Proof of Proposition \ref{P:bad_L1}]
Recall the definition \eqref{E:defMN} of $M_N$, $N\ge 0$, which depends on a small parameter $\eta>0$.
Let $\rho \in (0,1)$ be close enough to 1 so that $\rho^2(1+\eta) > 1$.
We fix an integer $\mathfrak{m} \ge 1$ and set $\delta = e^{-\mathfrak{m}}$.
Except when explicitly mentioned, all the constants in this proof will be independent of $\mathfrak{m}$ but may depend on $\gamma,\hat\gamma$, $\eta$ and $\rho$.
We are going to show the existence of two constants $C_1(\mathfrak{m}),C_2(\mathfrak{m})>0$ such that
\begin{equation}\label{E:pf_goal}
\lim_{\mathfrak{m}\to\infty} C_2(\mathfrak{m}) =0
\qquad \text{and} \qquad 
M_N \le \max (M_{N-1}, C_1(\mathfrak{m}) + C_2(\mathfrak{m}) M_{N-1}), \quad N \ge \mathfrak{m}.
\end{equation}
In particular, when $\mathfrak{m}$ is large enough, \eqref{E:pf_goal} ensures that $C_2(\mathfrak{m}) <1$. By Lemma \ref{L:very_elementary}, and because $M_N$ is finite for any $N$ (see Lemma \ref{L:initialisation}), this will imply that $(M_N)_{N \ge 0}$ is a bounded sequence.
We will also show that for all $N \ge \mathfrak{m}$,
\begin{equation}\label{E:pf_goal2}
    \sup_{k\ge0} \frac{1}{k!}\Big((1+\eta) \frac{\sqrt2}{\hat\gamma-\gamma}\Big)^{-k} \sup_{\af \in [0,1]} \sup_{0<\|f\|_\infty <\infty} \|f\|_\infty^{-1} \E[ |I_{\eps,\delta}^{\af,\bad}(f,k)|]
    \le C_2(\mathfrak{m}) M_{N-1}.
\end{equation}
Knowing that $(M_N)_{N\ge0}$ is uniformly bounded, this will thus prove that the supremum over $N \ge \mathfrak{m}$ of the left hand side converges to zero as $\mathfrak{m} \to \infty$, i.e. this will prove \eqref{E:P_bdd_L1} which will conclude the proof. The rest of the proof is dedicated to showing \eqref{E:pf_goal} and \eqref{E:pf_goal2}.


Let $N \ge \mathfrak{m}$ and $k \ge0$.
As in the introduction of Section \ref{S:convergence}, let $p=p(\gamma) \ge 1$ be the smallest integer such that $2p\gamma \ge \sqrt{2d}$.
We bound
\begin{align}\label{E:pf_supsup1}
\sup_{\af \in [0,1]} \sup_{0<\|f\|_\infty <\infty} \|f\|_\infty^{-1} \E[ |I_{\eps}^{\af}(f,k)|]
& \le \sup_{\af \in [0,1]} \sup_{0<\|f\|_\infty <\infty} \|f\|_\infty^{-1} \E[ |I_{\eps,\delta}^{\af,\good}(f,k)|^{2p}]^{1/(2p)} \\
& + \sup_{\af \in [0,1]} \sup_{0<\|f\|_\infty <\infty} \|f\|_\infty^{-1} \E[ |I_{\eps,\delta}^{\af,\bad}(f,k)|].
\notag
\end{align}
By Proposition \ref{P:bounded_L2} when $p=1$ and Proposition \ref{P:bound_pmoment} when $p \ge 2$, the first term is bounded by
\[ 
C \Big((1+\eta)\frac{\sqrt2}{\hat\gamma-\gamma} \Big)^k k! \delta^{-d/2} \le Ce^{\mathfrak{m}d/2}\Big((1+\eta)\frac{\sqrt2}{\hat\gamma-\gamma} \Big)^k k!.
\]
Lemma~\ref{L:inductive} deals with the second term: it is bounded by
\begin{align*}
    C \sum_{n=\mathfrak{m}}^N \sum_{j=0}^k \binom{k}{j} \sup_{\af \in [0,1]} \sup_{0<\|f\|_\infty <\infty} \|f\|_\infty^{-1} \E[\mathbf{1}_{N_n \ge (\hat\gamma-\gamma)n+C} |I_{e^n \eps}^{\af_n}(\tilde f_{k-j,n},j)| ],
\end{align*}
where $N_n$ is a centred Gaussian random variable with variance at most $n+C$ and $\tilde f_{k-j,n}$ is as in the statement of the lemma.
By definition of $(M_n)_{n \ge0}$, we can bound for all $n,j$,
\begin{align*}
    \E[\mathbf{1}_{N_n \ge (\hat\gamma-\gamma)n+C} |I_{e^n \eps}^{\af_n}(\tilde f_{k-j,n},j)| ]
    \le M_{N-n} \Big((1+\eta)\frac{\sqrt2}{\hat\gamma-\gamma} \Big)^j j! \E[\mathbf{1}_{N_n \ge (\hat\gamma-\gamma)n+C} \|\tilde f_{k-j,n}\|_{\infty} ].
\end{align*}
Because $(M_n)_{n \ge0}$ is nondecreasing, for all $n = \mathfrak{m},\dots,N$, $M_{N-n} \le M_{N-1}$. Moreover, using the bound \eqref{E:L_inductive2} on $\|\tilde f_{k-j,n}\|_{\infty}$, we get that
\begin{align*}
    & \E[\mathbf{1}_{N_n \ge (\hat\gamma-\gamma)n+C} \|\tilde f_{k-j,n}\|_{\infty} ]
    \le C \|f\|_\infty n^{(k-j)/2} \rho^{-(k-j)} \sqrt{(k-j)!} \E[ \mathbf{1}_{N_n \ge (\hat\gamma-\gamma)n+C} e^{\frac{1+C/n}{4n} N_n^2} ] \\
    & \le C \|f\|_\infty n^{(k-j)/2} \rho^{-(k-j)} \sqrt{(k-j)!} e^{-\frac{(\hat\gamma-\gamma)^2}{4}n}.
\end{align*}
Putting things together, we deduce that the second right hand side term of \eqref{E:pf_supsup1} is at most
\begin{align*}
    C M_{N-1} \sum_{n=\mathfrak{m}}^N \sum_{j=0}^k \binom{k}{j} \Big((1+\eta)\frac{\sqrt2}{\hat\gamma-\gamma} \Big)^j j! n^{(k-j)/2} \rho^{-(k-j)} \sqrt{(k-j)!} e^{-\frac{(\hat\gamma-\gamma)^2}{4}n}.
\end{align*}
We now exchange the sums to sum over $n$ first. By Lemma \ref{L:elementary_sum} (applied to $\eta = 1/\rho -1$), we can bound
\[
\sum_{n \ge \mathfrak{m}} e^{-\frac{(\hat\gamma-\gamma)^2}{4} n} n^{(k-j)/2} \le C ((k-j)/2)! \Big( \frac1\rho \frac{2}{\hat\gamma-\gamma} \Big)^{k-j} e^{-c \mathfrak{m}}.
\]
After elementary simplifications, we obtain that the second right hand side term of \eqref{E:pf_supsup1} is at most
\begin{align*}
    C e^{-c\mathfrak{m}} M_{N-1} \Big((1+\eta)\frac{\sqrt2}{\hat\gamma-\gamma} \Big)^k k! \sum_{j=0}^k \frac{((k-j)/2)!}{\sqrt{(k-j)!}} \Big( \frac1{\rho^2(1+\eta)} \sqrt{2} \Big)^{k-j}.
\end{align*}
By Stirling's formula, $\frac{((k-j)/2)!}{\sqrt{(k-j)!}} 2^{(k-j)/2}$ behaves at most polynomially. Because we have taken $\rho$ such that $\rho^2(1+\eta) > 1$, the above sum over $j$ is bounded by some constant. This proves \eqref{E:pf_goal2}. Going back to \eqref{E:pf_supsup1}, we have obtained that
\begin{align*}
    M_N \le C e^{\mathfrak{m}d/2} + C e^{-c \mathfrak{m}} M_{N-1}.
\end{align*}
This concludes the proof of \eqref{E:pf_goal} and the proof of Proposition \ref{P:bad_L1}.
\end{proof}

\section{Proof of Theorem \ref{T:main}}\label{S:reduction}

In this section, we use the coupling result \cite[Theorem B]{zbMATH07172346} to prove Theorem \ref{T:main} from the results derived in Section~\ref{S:convergence}.
Let $D\subset\R^d$ be a domain and $X$ be a centred log-correlated Gaussian field on $D$ as above Theorem~\ref{T:main}. Let $f : D\to \R$ be a bounded measurable function, whose support $\mathrm{supp}(f)$ is compactly included in $D$.
By \cite[Theorem B]{zbMATH07172346}, one can cover $\mathrm{supp}(f)$ by finitely many open balls $B_i$, $i=1, \dots, n$, and such that the following holds. For all $i=1, \dots, n$, one can decompose in $B_i$,
\begin{equation}
    \label{E:decompo_log}
X = X^* + R
\end{equation}
where $X^*$ is an almost $*$-scale invariant field and $R$ is a regular Gaussian field with Hölder continuous realisations (both depending possibly on $i$). Moreover, $X^*$ and $R$ are independent.
Letting $\xi_i : D \to \R, i=1, \dots, n$, be smooth functions such that $\sum \xi = 1$ on $\mathrm{supp}(f)$ and such that for all $i=1, \dots, n$, $\mathrm{supp}(\xi_i) \subset U_i$, we can decompose
\[
f = \sum_{i = 1}^n \xi_i f.
\]
By considering $\xi_i f$ instead of $f$, we may assume without loss of generality that the decomposition \eqref{E:decompo_log} holds on the whole support of $f$. Moreover, without loss of generality, we may assume that the support of $f$ is included in $[0,1]^d$.

Now, let $\gamma \in (-\sqrt{2d},\sqrt{2d})$, $k \ge 0$, $\eps >0$, $\varphi:\R^d \to \R$ be a smooth mollifier and $X_\eps = X * \varphi_\eps$ where $\varphi_\eps = \eps^{-d} \varphi(\cdot/\eps)$ and similar notations for $X^*$ and $R$.
Let us first show that for all $\eta>0$, there exists $C>0$ independent of $k$, $\eps$ and $f$, such that
\begin{equation}
    \label{E:reduction1}
    \E\Big[\Big| \int_D f(x) \wick{X_\eps(x)^k e^{\gamma X_\eps(x)}} \d x \Big| \Big] 
    \le C \|f\|_\infty k! \Big( (1+\eta) \frac{\sqrt{2d}-\gamma}{\sqrt2} \Big)^k.
\end{equation}
The decomposition \eqref{E:decompo_log} implies that $X_\eps = X^*_\eps + R_\eps$ on $\mathrm{supp}(f)$. By the umbral identity \eqref{E:L_Hermite_umbral1},
\begin{align}\label{E:reduction2}
    f(x) \wick{X_\eps(x)^k e^{\gamma X_\eps(x)}}
    = \sum_{j=0}^k \binom{k}{j} f(x) \wick{X_\eps^*(x)^j e^{\gamma X_\eps^*(x)}} \times \wick{R_\eps(x)^{k-j} e^{\gamma R_\eps(x)}}.
\end{align}
Let $j\in \{0,\dots, k\}$.
By Propositions \ref{P:bounded_L2} and \ref{P:bad_L1} applied to the function $x \mapsto f(x) \wick{R_\eps(x)^{k-j} e^{\gamma R_\eps(x)}}$, we have
\begin{align*}
    & \E\Big[\Big| \int_D f(x) \wick{X_\eps^*(x)^j e^{\gamma X_\eps^*(x)}} \times \wick{R_\eps(x)^{k-j} e^{\gamma R_\eps(x)}} \d x \Big| \Big] \\
    & \le C j! \Big( (1+\eta) \frac{\sqrt{2d}-\gamma}{\sqrt2} \Big)^j \E[ \sup_{x \in D} \big| f(x) \wick{R_\eps(x)^{k-j} e^{\gamma R_\eps(x)}} \big| ].
\end{align*}
Because $R$ is a continous Gaussian field and by Lemma \ref{L:Hermite_bound} which gives bounds on the Hermite polynomials, the last expectation above can be bounded by
\[
C \|f\|_\infty C^{k-j} \sqrt{(k-j)!}.
\]
We thus obtain that
\begin{align*}
    \E\Big[\Big| \int_D f(x) \wick{X_\eps(x)^k e^{\gamma X_\eps(x)}} \d x \Big| \Big] 
    & \le C \|f\|_\infty k! \Big( (1+\eta) \frac{\sqrt{2d}-\gamma}{\sqrt2} \Big)^k \sum_{j=0}^k \frac{1}{\sqrt{(k-j)!}} C^{k-j}\\
    & \le C \|f\|_\infty k! \Big( (1+\eta) \frac{\sqrt{2d}-\gamma}{\sqrt2} \Big)^k.
\end{align*}
This gives the uniform bound \eqref{E:reduction1} we were after.

To show that $\int_D f(x) \wick{X_\eps(x)^k e^{\gamma X_\eps(x)}} \d x$ converges in $L^1$, we first replace $R_\eps$ by $R$: by \eqref{E:reduction2},
\begin{align*}
    & \E \Big[ \Big| \int_D f(x) \wick{X_\eps(x)^k e^{\gamma X_\eps(x)}} \d x - \sum_{j=0}^k \binom{k}{j} f(x) \wick{X_\eps^*(x)^j e^{\gamma X_\eps^*(x)}} \times \wick{R(x)^{k-j} e^{\gamma R(x)}}  \Big| \Big] \\
    & \le C \sum_{j=0}^k \binom{k}{j} \Big( \int_D |f(x)|\E[|\wick{X_\eps^*(x)^j e^{\gamma X_\eps^*(x)}}| ] \d x\Big)\\
    & \hspace{40pt} \times
    \E\Big[ \sup_{x \in \mathrm{supp}(f)} |\wick{R_\eps(x)^{k-j} e^{\gamma R_\eps(x)}} - \wick{R(x)^{k-j} e^{\gamma R(x)}}| \Big].
\end{align*}
Because $R$ is Hölder continuous, the last expectation on the right hand side decays polynomially in $\eps$. On the other hand, the integral on the right hand side blows up at most like a power of $|\log \eps|$ (coming from $\var(X_\eps^*(x))^{j/2}$ in front of the Hermite polynomial $H_j$). Hence the right hand side goes to zero. But then $x \mapsto f(x) \wick{R(x)^{k-j} e^{\gamma R(x)}}$ is simply a test function and so, by Theorem~\ref{T:convergence_k}, conditionally on $R$,
\begin{align*}
    \sum_{j=0}^k \binom{k}{j} f(x) \wick{X_\eps^*(x)^j e^{\gamma X_\eps^*(x)}} \times \wick{R(x)^{k-j} e^{\gamma R(x)}}
\end{align*}
converges in $L^1(\P(\cdot\vert R))$. Convergence in $L^1(\P)$ then follows from domination estimates with a similar flavour as above. This concludes the proof of convergence of $\int_D f(x) \wick{X_\eps(x)^k e^{\gamma X_\eps(x)}} \d x$ in $L^1$ as stated in \eqref{E:T_main1}. The proof of the convergence of \eqref{E:T_main3} then follows by dominated convergence theorem, the domination being provided by the uniform estimate \eqref{E:reduction1}. This concludes the proof of Theorem \ref{T:main}.
\qed

\medskip

\noindent
\textbf{Acknowledgements.}
Many thanks to Guillaume Baverez, Baptiste Cerclé and Hubert Lacoin for inspiring discussions. Some of this work was done while the author attended the trimester programme \textit{Probabilistic methods in quantum field theory} at the Hausdorff Research Institute for Mathematics, funded by the Deutsche Forschungsgemeinschaft (DFG, German Research Foundation) under Germany's Excellence Strategy – EXC-2047/1 – 390685813
and the workshop \textit{CFT: algebraic, topological and probabilistic approaches to correlators in conformal field theory} at the Institut Pascal at Université Paris-Saclay with the support of the program ``Investissements d’avenir'' ANR-11-IDEX-0003-01. Thanks to these institutes and to the organisers for their hospitality.

\bibliographystyle{alpha}
\bibliography{bibliography}
\end{document}